\documentclass[11pt]{article}

\usepackage{amssymb,amsmath,amsthm, graphicx}
\usepackage{pgf,tikz}
\usetikzlibrary{arrows}

\newtheorem{theorem}{Theorem}[]
\newtheorem{cor}{Corollary}
\newtheorem{lemma}{Lemma}

\newcommand{\p}{\mathbf{P}}
\newcommand{\E}{\mathbf{E}}
\renewcommand{\d}{\mathrm{d}}
\newcommand{\e}{\mathrm{e}}
\newcommand{\R}{{\mathbb{R}}}
\newcommand{\N}{\mathbb{N}}
\newcommand{\0}{{\mathbf{0}}}
\newcommand{\x}{{\mathbf{x}}}
\newcommand{\m}{{\mathbf{m}}}
\renewcommand{\l}{{\mathbf{\ell}}}
\newcommand{\ev}{{\mathbf{e}}}
\newcommand{\unit}{\mathbf{1}}

\renewcommand{\c}{\mathbf{c}}

\renewcommand{\v}{\mathbf{v}}
\newcommand{\y}{\mathbf{y}}
\newcommand{\X}{\mathbf{X}}
\newcommand{\Y}{\mathbf{Y}}
\newcommand{\stleq}{\leq_{\mathrm{st}}}
\newcommand{\stgeq}{\geq_{\mathrm{st}}}

\renewenvironment{proof}{\smallskip \noindent \textbf{Proof.}}
{\vspace*{1pt} \hfill $\square$ \medskip}

\newenvironment{proof*}[1]{\smallskip \noindent
\textbf{Proof of #1.}}
{\vspace*{1pt} \hfill $\square$ \medskip}

\begin{document}

\title{On the diminishing process of B. T\'oth}

\author { \textsc{P\'eter Kevei}\thanks{This research was supported by the European Union and the State of Hungary,
co-financed by the European Social Fund in the framework of T\'AMOP-4.2.4.A/ 2-11/1-2012-0001  ‘National Excellence Program’.} \\
MTA-SZTE Analysis and Stochastic Research Group \\ Bolyai Institute, University of Szeged \\
\texttt{kevei@math.u-szeged.hu}
\and \textsc{Viktor V\'igh}\footnotemark[1] \\
Department of Geometry \\ Bolyai Institute, University of Szeged \\
\texttt{vigvik@math.u-szeged.hu}
}

\date{\today}

\maketitle

%    Abstract is required.
\begin{abstract}
 Let $K$ and $K_0$ be  convex bodies in $\R^d$, such that
$K$ contains the origin, and  define the process $(K_n, p_n)$,
$n \geq 0$, as follows: let $p_{n+1}$ be a uniform
random point in $K_n$, and set $K_{n+1} = K_n \cap (p_{n+1} + K)$. Clearly,
$(K_n)$ is a nested sequence of convex bodies which converge to a non-empty
limit object, again a convex body in $\R^d$. We study this process for $K$ being a regular simplex, a cube,
or a regular convex polygon with an odd number of vertices.
We also derive some new results in one dimension for non-uniform distributions.
\end{abstract}

\section{Introduction}

The following problem was formulated by B\'alint
T\'oth some 20 years ago with $K=K_0$ being the unit disc of the plane. 
Let $K$ and $K_0$ be  convex bodies in $\R^d$, such that
$K$ contains the origin, and  define the process $(K_n, p_n)$,
$n \geq 0$, as follows: let $p_{n+1}$ be a uniform
random point in $K_n$, and set $K_{n+1} = K_n \cap (p_{n+1} + K)$. Clearly,
$(K_n)$ is a nested sequence of convex bodies which converge to a non-empty
limit object, again a convex body in $\R^d$. What can we say about the
distribution of this limit body? What can we say about the speed of the process?
In Figure \ref{fig:7gon} one can see the evolution of the process up to $n=10$ on the right,
and $K_{10}$ on the left, when $K = K_0$ is a regular heptagon.
% When $K$ is the unit disc the limit object is almost surely a convex disc of constant
% width 1.

\begin{figure} 
\begin{center}
\includegraphics[scale=0.38, angle=180]{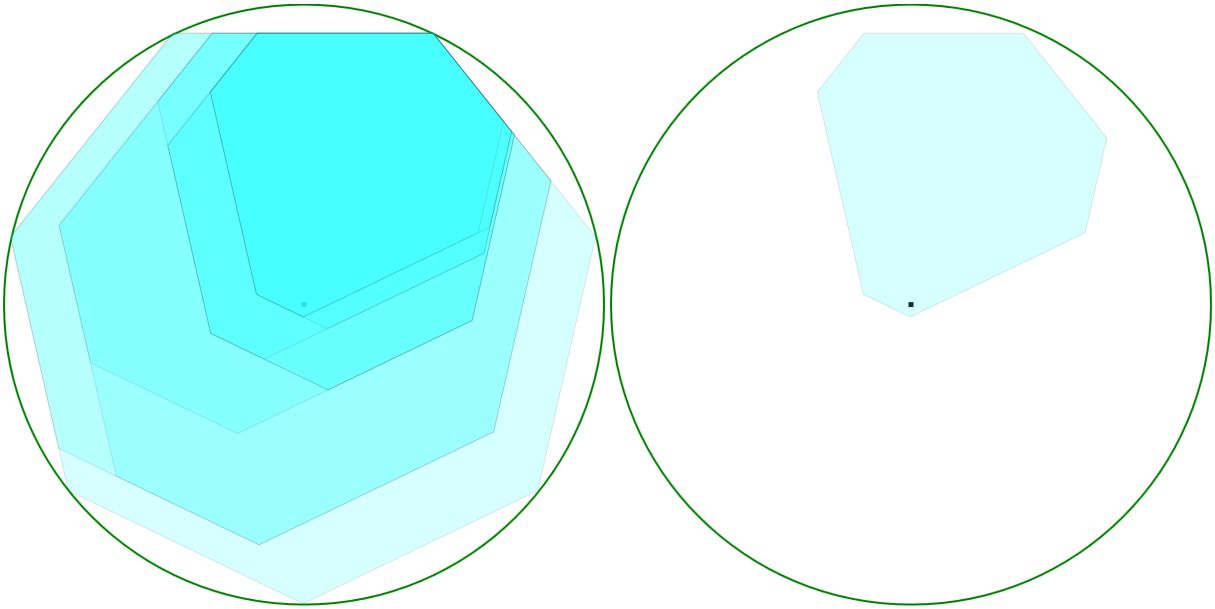}
\caption{\label{fig:7gon} The evolution of the process for $K=K_0$ being a regular heptagon}
\end{center}
\end{figure}

In \cite{AKV} Ambrus, Kevei and V\'igh investigated the process in $1$ dimension, when
$K = K_0 = [-1,1]$. In this case the limit object is a random unit interval, whose
center has the arcsine distribution (see Theorem $1$ in \cite{AKV}). 
So even in the simplest case the process has very interesting features.
Moreover, in Theorem $2$ in \cite{AKV}
it is shown that if $r_n$ is the radius of the interval $K_n$, then $ 4n (r_n - 1/2)$
converges in distribution to a standard exponential random variable.
The idea of the proof is to observe that $(r_n - 1/2)$ behaves as the minimum
of iid random variables, and thus obtain the limit theorem via extreme value theory.

We also would like to point out the formal relationship between the diminishing process and the so called R\'enyi's Parking Problem from 1958 \cite{renyi58}.
R\'enyi studied the following random process:
consider an interval $I$ of length $x>>1$, and sequentially and randomly pack (non-overlapping) unit intervals into $I$.
In each step we choose the center of the next unit interval uniformly from the possible space.
The process stops when there is no space for placing a new unit interval. (Intuitively $I$ is the parking lot and the unit
intervals are the cars.) The first possible question is to determine the expectation $M(x)$
of the covered space. Many other variants of this problem has been studied in the last more than 50 years, for an up-to-date
state of the art we refer to Clay and Sim\'anyi \cite{sim}. The connection between the diminishing process and
R\'enyi's Parking Problem can be seen easily as follows: if we choose in the definition of the diminishing process $K_0=I$,
and we drop the conditions we put on $K$, 
and define $K$ as the complement of
the closed interval of length $2$ centered at $0$, then we get exactly R\'enyi's Parking Problem.

In the present paper we analyze the diminishing process in more general cases. In Section \ref{sect:1dim}
we consider the case, when instead of choosing $p_{n+1}$ uniformly in the interval, we
choose it according to a translated and scaled version of a fixed distribution $F$.
Again, the limit object is a random unit interval.
In Theorem \ref{thm:1dim-rapid} we determine the asymptotic behavior of the speed, while
in Theorem \ref{thm:1dim-center} we show that for appropriate choice of $F$ the distribution
of the center has the beta law.
In Sections  \ref{sect:cube} and \ref{sect:simplex} we consider the case when $K = K_0$ is a cube and a regular $d$-dimensional
simplex, respectively.  The cube process can be represented as $d$ independent interval
processes, thus the results in Section \ref{sect:cube} follow from the corresponding results in \cite{AKV}.
In the case of the simplex process, the limit object is also a random regular simplex. The main result of this part is that
the center of the limit simplex in barycentric coordinates has multidimensional Dirichlet law,
which is a natural generalization of the beta laws to any dimension. The rate of the process is also
determined. The processes considered this far are `self-similar' in the sense that at each step the process is
a scaled and translated version of the original one.

In Sections \ref{sect:5}, \ref{sect:pentagon} and \ref{sect:sok} we consider diminishing 
processes in the plane. In case of the pentagon process even the shape of the limiting object is random. We prove that it is
a pentagon with equal angles, however it is not regular a.s. This process is not `self-similar', and its 
behavior is more complicated. We determine the rate of the convergence of the maximal height,
but as the area of the limit object is random, limit theorem with deterministic normalization is not possible.
Also the behavior of the center of mass is intractable with our methods. Finally, in Section \ref{sect:sok}
we consider regular polygons with odd number of vertices, i.e.~$K= K_0$ is  a regular polygon. Using the theory
of stochastic orderings for random vectors we prove that the rate of the speed is $n^{-1/2}$.
We conjecture that in the case, when the number of vertices is even the speed of the process is $n^{-1}$.
This is established in case of square, but in general
it is open.

\section{One dimension, general density} \label{sect:1dim}

In this section we consider the process in the interval $[-1,1]$, and the random point is chosen
according to a not necessarily uniform distribution.

Fix a distribution on $[0,1]$ with distribution function $F$, and in each
step we choose the random point according to this distribution. That is,
if the center and radius is $(Z_n, r_n)$ the random point $p_{n+1}$ is
given by $2 r_n X_{n+1} + Z_n - r_n$, where $X_{n+1}$ is independent
from $Z_n, r_n$, and has distribution function $F$. The initial condition
is $(Z_0, r_0) = (0,1)$, i.e.~we start from the interval $[-1,1]$.

Let $X, X_1, X_2, \ldots$ iid random variables with distribution function $F$. It is easy to see that
for $n \geq 0$
\begin{equation} \label{eq:gen-rho-recursion}
r_{n+1} = \left\{
\begin{array}{ll}
\frac{1}{2} + r_n \min \{ X_{n+1}, 1- X_{n+1} \}, & 
\min \{ X_{n+1}, 1- X_{n+1} \} \leq 1 - \frac{1}{2 r_n}, \\
r_n, & \textrm{otherwise.}
\end{array} \right.
\end{equation}

To simplify the recursions above we have to pose some assumptions on $F$.
The following lemmas contain these assumptions. To determine the rapidness
of the process we only need part (i), while for the limit distribution of the
center we need both parts. In fact, in both cases we only need the `if' part.
In the following, for a random variable $X$ and an event $A$ the notation
 $X | A$ stands for the conditional distribution of $X$ given $A$.

\begin{lemma} \label{lemma:konnyu}
Let $X$ be a random variable with distribution function $F$, such that $\p \{ X \in [0,1] \} = 1$.
\begin{itemize}
\item[(i)] For all $a \in [0,1]$, for which $\p \{ X \leq a \} > 0$, the distributional equality
\[
X | ( X \leq a) \stackrel{\mathcal{D}}{=} a X
\]
holds, if and only if either $X$ is a degenerate random variable at 0, or
$F(x) = x^\delta$, $x \in [0,1]$, for some $\delta > 0$.

\item[(ii)] The random variables $I(X \leq 1/2)$ and $\max \{ X, 1- X\}$ are
independent if and only if $F(1/2) = 0$, or $F(1/2) = 1$, or
\[
F(x) = 1 - \frac{1- F(1/2)}{F(1/2)} F(1- x -)
\]
for all $x \in [1/2, 1]$.
\end{itemize}
\end{lemma}

The simple proof of Lemma~\ref{lemma:konnyu} is given in the Appendix.
As an immediate consequence we obtain the following.

\begin{lemma} \label{lemma:corollary}
Let $Y$ be a random variable in $[0,1]$ with continuous distribution function $F$.
Then for any $a \in (0,1)$ the distributional equality
\[
2 \min \{ Y, 1- Y \} | \left(  2 \min \{ Y, 1- Y \} \leq a \right)
\stackrel{\mathcal{D}}{=} a \, 2 \min \{ Y, 1- Y \} 
\]
holds, and $I(Y \leq 1/2)$ and $\max \{ Y, 1- Y\}$ are
independent if and only if
\begin{equation} \label{eq:df-form}
F(x) = \left\{
\begin{array}{ll}
c 2^\delta x^\delta, & x \in [0,1/2], \\
1 - (1-c) 2^\delta (1-x)^\delta, & x \in [1/2, 1],
\end{array} \right. 
\end{equation}
for some $c \in [0,1]$ and $\delta > 0$.
\end{lemma}

During the analysis of diminishing processes 
we frequently end up with a recursion of the following type.

Let $V, V_1, \ldots$ be a sequence of iid random variables with distribution function
$\p \{ V \leq x \} = x^\delta$, $x \in [0,1]$, for some $\delta > 0$,
and let $(a_n)$ be a sequence of bounded nonnegative random variables, such that 
$a_n \downarrow a$, a.s., where $a > 0$ is deterministic. 
Assume that $\ell_0= 1$, and for $n \geq 0$, for some $c > 0$
\begin{equation} \label{eq:ell-recursion}
\ell_{n+1} = 
\begin{cases}
 \ell_n V_{n+1}, & \textrm{w.p. } c \frac{\ell_n^\delta}{a_n}, \\
 \ell_n, & \textrm{w.p. } 1 - c \frac{\ell_n^\delta}{a_n},
\end{cases}
\end{equation}
where $c \frac{\ell_n^\delta}{a_n} \in [0,1]$ and
the abbreviation w.p.~stands for `with probability'.

To be precise this means here and later on the following.
On our probability space $(\Omega, \mathcal{A}, \p)$ there is
a filtration $(\mathcal{F}_n )_{n \geq 0}$. The filtration is usually
generated by the random points $p_n$, i.e.~$\mathcal{F}_n = \sigma(p_1, \ldots, p_n)$.
The random variables $a_n$ and $\ell_n$ are $\mathcal{F}_n$ measurable, and
almost surely $a_n \downarrow a > 0$. Conditionally on $a_n$ and $\ell_n$
let $\omega_{n+1}$ be a Bernoulli$(c\frac{\ell_n^\delta}{a_n})$ random variable and
independently $V_{n+1}$ is a random variable with distribution function 
$x^\delta$, $x \in [0,1]$. Then $\ell_{n+1} = \ell_n V_{n+1}$ whenever
$\omega_{n+1} = 1$, and $\ell_{n+1} = \ell_n$ otherwise.
(Here and in the following section $a_n$ is simply a function of $\ell_n$. However,
when dealing with the polygon process $a_n$ is the area of $K_n$, and it does depend
on the chosen points, and not only on $\ell_n$. This is the reason of the complication.)

In the next lemma we determine the asymptotic behavior of such sequence $\ell_n$.
The idea of the proof is to show that $\ell_n$ behaves like the minimum of
$n$ iid random variables, as in the proof of Theorem 1 in \cite{AKV}.
The proof is deferred to the Appendix.

For $\delta > 0$, the Weibull$(\delta)$ distribution function
is given by $1 - \e^{-x^{\delta}}$, for $x > 0$, and 0 otherwise.

\begin{lemma} \label{lemma:recursion}
Assume that $\ell_n$ is defined by (\ref{eq:ell-recursion}). Then
\[
\left( \frac{c}{a} n \right)^{1/\delta} \ell_n
\stackrel{\mathcal{D}}{\longrightarrow} \mathrm{Weibull}(\delta).
\]
Moreover, for any $\alpha > 0$
\[
\lim_{n \to \infty}
\E \left( \frac{c}{a} n \right)^{\alpha/\delta} \ell_n^\alpha
=\frac{\alpha}{\delta} \Gamma \left( \frac{\alpha}{\delta} \right),
\]
where $\Gamma(\, \cdot \,)$ is the usual Gamma-function.
\end{lemma}

With the help of these lemmas we can analyze the speed of the process.

\begin{theorem} \label{thm:1dim-rapid}
Assume that for the distribution of $X$ we have
\[
\p \{ 2 \min \{ X, 1 -X \} \leq x \} = x^\delta, \quad x \in [0,1],
\]
for some $\delta >0$. Then as $n \to \infty$
\[
4 n^{1/\delta} \left( r_n - \frac{1}{2} \right) \stackrel{\mathcal{D}}{\longrightarrow}
\mathrm{ Weibull}(\delta),
\]
i.e.~for any $x > 0$
\[
\p \left\{  4 n^{1/\delta} \left( r_n - \frac{1}{2} \right) > x \right\}
\to \e^{-x^\delta}.
\]
Moreover, for any $\alpha >  0$
\[
\lim_{n \to \infty} \E 4^{\alpha} n^{\alpha / \delta} \left( r_n - \frac{1}{2} \right)^{\alpha}
= \frac{\alpha}{\delta} \Gamma \left( \frac{\alpha}{\delta} \right).
\]
\end{theorem}

\begin{proof}
Using the assumption and Lemma \ref{lemma:konnyu} (i) we see that 
(\ref{eq:gen-rho-recursion}) can be rewritten as
\begin{equation} \label{spec-rho-recursion}
 \ell_{n+1} = \left\{
 \begin{array}{ll}
  \ell_n V_{n+1}, & \textrm{w.p. } (2 - r_n^{-1})^\delta, \\
  \ell_n, & \textrm{w.p. } 1 - (2 - r_n^{-1})^\delta,
 \end{array} \right.
\end{equation}
with $\ell_n = r_n - 1/2$, and $V, V_1, \ldots$ are iid, $\p \{ V \leq x \} = x^\delta$,
$x \in [0,1]$. Now the theorem follows from Lemma \ref{lemma:recursion},
with $a_n = r_n^\delta \downarrow 1/2^\delta = a$ and $c= 2^\delta$.
\end{proof}

To determine the limit distribution of the center
consider the thinned process $(\widetilde Z_n, \widetilde r_n)$,
which is obtained from the original process $(Z_n, r_n)$ by dropping those
steps when nothing changes, i.e.~when $r_n = r_{n+1}$. Clearly, the limit
of the center is not affected. After some
calculation we obtain the recursion
\begin{equation} \label{gen-recursion}
\begin{split}
\widetilde Z_{n+1} & = \widetilde Z_n +
\frac{2 \widetilde r_n \max \{ X_{n+1}, 1 - X_{n+1} \} - 1}{2}  
\mathrm{sgn} \, (X_{n+1} - 1/2),
\\
  \widetilde r_{n+1} & = \frac{1}{2} + \widetilde r_n \min \{ X_{n+1}, 1 - X_{n+1} \}, 
 \end{split}
\end{equation}
where $X_{n+1}$ has the distribution of $X$ conditioned on 
$\min \{ X, 1- X \} < 1 -  (2 \widetilde r_n)^{-1}$.

Note that in (\ref{eq:df-form}) in Lemma \ref{lemma:corollary}
for $c=1$ the distribution
is concentrated on $[0,1/2]$, in which case the center always moves
towards $-1/2$, so the limit distribution of the center is degenerate at
$-1/2$. Similarly, for $c=0$ the limit is deterministic $1/2$. In the following
theorem we exclude these cases. 

For $\alpha > 0, \beta > 0$ the random variable $X$ has  beta$(\alpha, \beta)$ 
law, if its density is $x^{\alpha - 1} (1 - x)^{\beta- 1} B(\alpha, \beta)^{-1}$,
$x \in (0,1)$, where $B(x, y)=\Gamma(x)\Gamma(y)/\Gamma(x+y)$ is the usual Beta function.

\begin{theorem} \label{thm:1dim-center}
Let us assume that for some $c \in (0,1)$ and $\delta > 0$
(\ref{eq:df-form}) holds. 
Then the distribution of $Z$ is
the translated beta$(\delta(1-c), \delta c)$ law, i.e.~its density function is
\[
f_{\delta, c}(x) = \frac{\Gamma(\delta)}{\Gamma(\delta(1-c)) \Gamma(\delta c)}
(1/2 + x)^{\delta (1-c) -1} (1/2 - x)^{\delta c -1}, \quad x \in (-1/2, 1/2).
\]
\end{theorem}

\begin{proof}
By Lemma \ref{lemma:konnyu}
and (\ref{gen-recursion}) we obtain the recursion
\begin{equation} \label{eq:spec-recursion}
 \begin{split}
  \widetilde Z_{n+1} & = \widetilde Z_n +  \xi_{n+1} \widetilde \ell_n ( 1- V_{n+1} ), \\
  \widetilde \ell_{n+1} & = \widetilde \ell_n V_{n+1}, 
 \end{split}
\end{equation}
where $\widetilde \ell_n = \widetilde r_n -1/2$, and $\xi_1, \xi_2, \ldots$ are iid random
variables, such that $\p \{ \xi_1 = 1 \} = 1- c = 1- \p \{ \xi_1 = -1 \}$, and
independently from $\{ \xi_i \}_{i=1}^\infty$, the sequence $V_1, V_2, \ldots$
are iid $\beta(\delta,1)$ random variables, i.e.~with distribution function
$\p \{ V \leq x \}  = x^\delta$.
The initial value is $(\widetilde Z_0, \widetilde \ell_0) = (0,1/2)$.

Formula (\ref{eq:spec-recursion}) implies the infinite series representation of the
limit
\begin{equation} \label{eq:infinite-repr}
Z_\infty= \frac{1}{2} \sum_{i=1}^\infty \xi_i V_1 \ldots V_{i-1} (1- V_i),
\end{equation}
and thus the distributional equation perpetuity
\begin{equation} \label{eq:perpet}
Z_\infty \stackrel{\mathcal{D}}{=} \frac{1}{2} \xi_1 ( 1 - V_1) + V_1 Z_\infty,
\end{equation}
where on the right-hand side $V_1, \xi_1, Z_\infty$ are independent.

Corollary 1.2 in Hitczenko and Letac \cite{HL} (or the proof of Theorem 3.4 in Sethuraman \cite{Seth})
implies that $Z_\infty +1/2$ has
$\beta( \delta (1-c), \delta c)$ distribution.
\end{proof}

Note that once we have the infinite series representation (\ref{eq:infinite-repr})
the proof can be
finished using the properties of \textrm{GEM}$(\delta)$ (or Poisson--Dirichlet)
law; see Hirth \cite{Hirth}, or Bertoin \cite{bertoin} Section 2.2.5.

\medskip

Distributional equations of type 
\[
R \stackrel{\mathcal{D}}{=} Q + MR, \quad R \text{ independent of } (Q,M),
\]
where $R, Q$ are random vectors, and $M$ is a random variable,
are called \textit{perpetuities}. Equation (\ref{eq:perpet}) is an example. 
Necessary and sufficient conditions for the existence of a unique solution
of one-dimensional perpetuities is given by Goldie and Maller \cite{goldiemaller}.
However, in special cases (for example for $M \in [-1,1]$)
the existence of a unique solution in any dimension was known earlier,
see Lemma 3.3 by Sethuraman \cite{Seth}. Therefore, in (\ref{eq:perpet}) above,
or in $d$ dimension in (\ref{eq:d-perpet}) below, the assertion that certain
distribution $G$ satisfies the perpetuity equation is equivalent to saying
that the perpetuity equation has a unique solution $G$.

The perpetuities (\ref{eq:perpet}) and (\ref{eq:d-perpet}) are interesting
in their own right, because there are relatively few perpetuities when the exact
solution is known. The results of Sethuraman \cite{Seth} (proof of Theorem 3.4;
see also Theorem 1.1 in \cite{HL}) cover those equations
which appear in our investigations. For more general perpetuity equations
with exact solutions we refer to the recent paper by Hitczenko and Letac \cite{HL}.

\section{The cube} \label{sect:cube}

In the cube process $K = K_0 = [-1,1]^d$. Now the limiting
convex body is a cube of unit edgelength. Denote $m_1(n), \ldots, m_d(n)$ the 
edgelengths of the rectangular box $K_n$, and $(Z_1(n),\ldots,  Z_d(n))$ the center
of $K_n$. Properties of the uniform distribution imply
that the processes $(Z_1(n), m_1(n))$,$\ldots$, $(Z_d(n), m_d(n))$ are $d$ independent copies of the segment process. Therefore
the following theorem is a consequence of Theorem 1 and 2 in \cite{AKV}.

\begin{theorem} \label{th:rectangle}
For the speed of the cube process we have
\[
2n \,
\begin{pmatrix}
m_1(n ) - 1 \\\vdots \\ m_d(n) - 1 
\end{pmatrix}
\stackrel{\mathcal{D}}{\longrightarrow}
\begin{pmatrix}
W_1 \\ \vdots \\ W_d
\end{pmatrix},
\]
where $W_1, \ldots, W_d$ are independent exponential random variables with parameter 1.
For the maximum of the edgelengths $m_n = \max\{ m_1(n),\ldots,  m_d(n) \}$ we have
\[
2 n (m_n - 1) \stackrel{\mathcal{D}}{\longrightarrow} W, 
\]
where $\p \{ W \leq x \} = (1 - \e^{-x} )^d$, $x \geq 0$.

For the limit distribution of the center
\[
\begin{pmatrix}
Z_1(n) \\ \vdots \\ Z_d (n)  
\end{pmatrix}
\stackrel{\mathcal{D}}{\longrightarrow} 
\begin{pmatrix}
 Z_1 \\ \vdots \\ Z_d 
\end{pmatrix},
\]
where $Z_1, \ldots, Z_d$ are independent translated arcsine random variables, that is with density
function
\[
\frac{1}{ \pi \sqrt{(1/2 +x)(1/2-x)}}, \quad x \in \left( -\frac{1}{2} , \frac{1}{2} \right).
\]
\end{theorem}

{\it Remark.} Similarly, the results obtained in Section~\ref{sect:1dim} can be generalized for a `non-uniform cube process'. The details are left to the interested reader.

\section{The simplex}\label{sect:simplex}

Now we turn to the simplex process in any dimension.

Let $K$ be a regular $d$-dimensional simplex with centroid $(0,0, \ldots, 0)$ and
vertices $(\e_0, \e_1, \ldots, \e_d)$, such that $\e_0 =  (1, 0, \ldots, 0).$
Let denote $\rho_d = 1/d$ the radius of the inscribed sphere of $K$.

Let the initial simplex be $K_0 =\frac{2}{d+1} K$ (for reasons explained below),
and for $K_n$ given, choose a
random point $p_{n+1}$ uniformly in $K_n$ and let $K_{n+1} = K_n \cap (p_{n+1} + K)$.
Let $m_n$ denote the height of $K_n$. 
Then $K_n$ is a nested sequence of regular simplicies and
the limit object is a regular simplex with height $\rho_d$.
% , which center falls in $\widehat K := (m_0 - \rho_d)/(1 + \rho_d) K$,
% i.e.~in a regular simplex with height $m_0 -\rho_d$.

It turns out that this process can be investigated by the same methods as
for $d=1$, in case of the segment process, in \cite{AKV}.
The idea is that for the simplex in any dimension
the process is `self-similar', i.e.~after each step the process is a translated
and scaled version of the original one.

\subsection{The rapidness of the process}

If in the $(n+1)$st step the point $p_{n+1}$ falls close to the center, then
nothing happens, i.e.~$K_{n+1} =  K_n$. The `change regions' are 
$d+1$ congruent, regular simplicies of height $m_n-\rho_d$, each of them sits at a vertex of $K_n$. Note that since the
height of $K_n$ is $\leq 2 \rho_d$ then these simplicies are disjoint, so the process
is simpler. This is the reason we assume $K_0 = \frac{2}{d+1} K$, since
its height $m_0 = 2 \rho_d$.
Although, if we would start with a larger $K_0$, as $m_n \downarrow \rho_d$ a.s., in a random number of steps
the height of $K_n$ would be $\leq 2 \rho_d$, thus the assumption $K_0 = \frac{2}{d+1} K$ has no effect on the rapidness
of the process.

\begin{theorem}
For the height process $m_n$
\[
\frac{(d+1)^{1/d}}{\rho_d} n^{1/d} (m_n - \rho_d)
\stackrel{{\mathcal{D}}}{\longrightarrow} \mathrm{Weibull}(d).
\]
Moreover, for any $\alpha > 0$
\[
\lim_{n \to \infty} \E \frac{[(d+1) n ]^{\alpha/d}}{\rho_d^\alpha} (m_n - \rho_d)^\alpha
= \frac{\alpha}{d} \Gamma \left( \frac{\alpha}{d} \right).
\]
\end{theorem}

\begin{proof}
With disjoint change regions for the height process we have
\[
m_{n+1} =
\begin{cases}
m_n - h_{n+1} \left( m_n - \rho_d \right) & \textrm{w.p. }
(d+1) \left( 1- \frac{\rho_d}{m_n} \right)^d, \\
m_n, & \textrm{w.p. } 1 - (d+1) \left( 1- \frac{\rho_d}{m_n} \right)^d,
\end{cases}
\]
where $h_1, h_2, \ldots$ are iid random variables, with distribution function
\begin{equation} \label{H_d}
H_d (x) = \p \{ h \leq x \} = 1 - \p \{ h > x \} = 1 - (1 - x)^d,
\quad x \in [0,1],
\end{equation}
which is the distribution of the distance from the
base of a uniformly distributed random point in a regular simplex with height 1,
see Figure \ref{fig:3szog}.

\begin{figure}
\begin{center}
\begin{tikzpicture}[line cap=round,line join=round,>=triangle 45,x=1.0cm,y=1.0cm]
\clip(-1.67,-0.3) rectangle (6.97,6.26);
\draw  (0,0)-- (6,0);
\draw  (6,0)-- (3,5.2);
\draw  (3,5.2)-- (0,0);
\draw  (1.75,3.03)-- (4.25,3.03);
\draw (-0.11,4.2) node {$m_n - \rho_2$};
\draw [->] (4.34,3.03) -- (4.34,3.79);
\draw [->] (4.34,3.79) -- (4.34,3.03);
\draw [->] (0.77,5.2) -- (0.77,3.03);
\draw [->] (0.77,3.03) -- (0.77,5.2);
\draw [->] (-0.69,3.03) -- (-0.69,0);
\draw [->] (-0.69,0) -- (-0.69,3.03);
\draw (-1.18,1.55) node {$\rho_2$};
\draw (5.8,3.6) node{$(m_n - \rho_2) h_{n+1}$};
\draw[color=black] (2.45,1.3) node {$K_n$};
\fill [color=black] (3.14,3.79) circle (1.5pt);
\draw[color=black] (3.1,4.1) node {$p_{n+1}$};
\end{tikzpicture} 
\caption{\label{fig:3szog}
 The triangle process}
\end{center}
\end{figure}
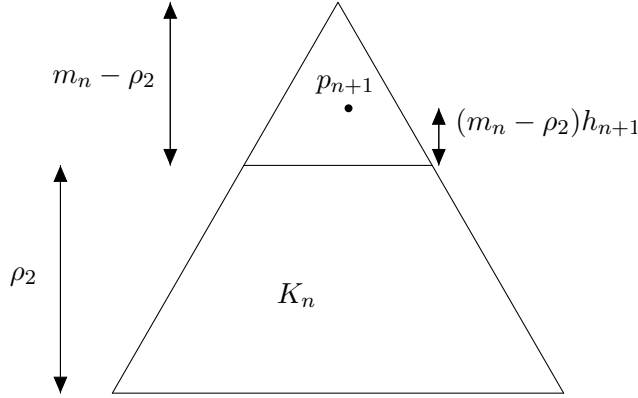

Putting $\ell_n = m_n - \rho_d$, we have $\ell_n \downarrow 0$ a.s., and
\begin{equation} \label{ell}
\ell_{n+1} =
\begin{cases}
\ell_n (1  - h_{n+1}) & \textrm{w.p. } (d+1) \left( 1- \frac{\rho_d}{m_n} \right)^d, \\
\ell_n, & \textrm{w.p. } 1 - (d+1) \left( 1- \frac{\rho_d}{m_n} \right)^d .
\end{cases}
\end{equation}
The theorem follows from Lemma \ref{lemma:recursion} with
$\delta = d, c = d+1$ and $a_n = m_n^d \downarrow \rho_d^d$.
\end{proof}

\subsection{The limit distribution of the center}

Let $\c_n$  denote the center of the regular simplex $K_n$. In this subsection
we determine the limit distribution of $\c_n$.

As we emphasized  previously the limit distribution
of $n^{1/d} (m_n - \rho_d)$ is not affected if we start from
any smaller regular simplex, in particular which has height $2 \rho_d$.
However, this is not true for the limit distribution of the center $\c_n$.
To handle the process we have to assume that the change regions are
disjoint, and so in each step the center can only move towards one of
the vertices, or stay. 

In order to investigate the limit distribution of the centroid, we can
consider the thinned (centroid, height) process $(\widetilde \c_n, \widetilde m_n)$,
skipping the steps when nothing happens. Put $\widetilde \ell_n = \widetilde m_n - \rho_d$.

Since the disjoint change regions have the same volume, in each step the center
moves towards any of the vertices with the same probability $1/(d+1)$, according
to the change region in which the chosen point falls.
The size of the shift is
$\frac{d}{d+1} \cdot \widetilde \ell_n h_{n+1}$, where
$\widetilde \ell_n h_{n+1}$ is the distance of the chosen point from the base of the change
region. See Figure \ref{fig:3szog}. Thus
\begin{equation}
\begin{split} \label{MC1-simp}
& \widetilde \c_{n+1} = \widetilde \c_n + \frac{d}{d+1} \widetilde \ell_n h_{n+1} \e_{\xi_{n+1}}, \\
& \widetilde \ell_{n+1} = \widetilde \ell_n (1 - h_{n+1}),
\end{split}
\end{equation}
where $h_1, h_2, \ldots $ are iid random variables with distribution function (\ref{H_d}),
$\xi_1, \xi_2, \ldots$ are
independent, uniformly distributed random variables on the set $\{0, 1, \ldots, d \}$,
and the initial condition is  $\widetilde \c_0 = \0$, $\widetilde \ell_0 = \rho_d$.

To obtain a more symmetric description of the center process we introduce
the barycentric coordinates. The center of the  limiting simplex
falls in $\widehat K := \frac{1}{d+ 1} K$,
i.e.~in a regular simplex with height $\rho_d$.

Put $\widehat \e_i =  \frac{1}{d+1} \e_i$,
that is $\widehat \e_0, \ldots, \widehat \e_d$ are the vertices of $\widehat K$. 
To parametrize the center we may use barycentric
coordinates in terms of $\widehat K$. That is, for $\c_n$ the
center of $K_n$, we have $\c_n = \sum_{i=0}^d \lambda_n^{i} \widehat \e_i$, with
$\sum_{i=0}^d \lambda_n^{i} =1$, $\lambda_n^{i} \geq 0$, $i=0,1, \ldots, d$.
It is well-known that this parametrization is unique.
Put $\Lambda_n = (\lambda_n^0, \ldots, \lambda_n^d) \in \R^{d+1}$.
We can rewrite (\ref{MC1-simp}) in terms of the barycentric coordinates of
$\widetilde \c_n$. After some calculation we have
\begin{equation} \label{MC1-simp-2}
\begin{split}
& \widetilde \Lambda_{n+1} = \widetilde \Lambda_n +
\frac{d}{d+1} \widetilde \ell_n h_{n+1} \v_{\xi_{n+1}}, \\
& \widetilde \ell_{n+1} = \widetilde \ell_n (1 - h_{n+1}),
\end{split} 
\end{equation}
where $\v_j$ is the constant $-1$ vector, except its $j$th coordinate
being $d$. The initial values are $\widetilde \Lambda_0 = (1/(d+1), \ldots, 1/(d+1))$,
$\widetilde \ell_0 = \rho_d$.

Before stating the theorem, we define the multidimensional Dirichlet distribution.
Let $a_0, \ldots, a_d$ be positive numbers. The random vector
$X = (X_0, \ldots, X_d)$ has Dirichlet$(a_0, \ldots, a_d)$ distribution,
if its components are nonnegative, $X_0 + \ldots + X_d = 1$, and
$(X_1, \ldots, X_d)$ has density function
\[
\frac{\Gamma(a_0 + \ldots + a_d)}{\Gamma(a_0) \ldots \Gamma(a_d)}
(1 - x_1 - \ldots -x_d)^{a_0 - 1} x_1^{a_1 - 1} \ldots
x_d ^{a_d - 1},
\]
on the set $\{ (x_1, \ldots, x_d): x_i \in (0,1), i=1,\ldots, d; \sum_{i=1}^d x_i \leq 1 \}$.

\begin{theorem} \label{thm:simlex-kp}
The barycentric coordinates of the center of the limit simplex
have $\mathrm{Dirichlet}(d/(d+1), \ldots, d/(d+1))$ distribution.
\end{theorem}

\begin{proof}
Let $\widetilde\Lambda$ be the barycentric coordinates of the center of the limit.
From (\ref{MC1-simp-2}) we obtain that
\[
\widetilde\Lambda =
\widetilde\Lambda_0 + \frac{1}{d+1} \sum_{n=0}^\infty (1 - h_1) \ldots (1 - h_n) h_{n+1} \v_{\xi_{n+1}}.
\]
Rearranging we get
\[
\begin{split}
\widetilde\Lambda & = h_1 \left( \frac{1}{d+1} \v_{\xi_1} + \widetilde\Lambda_0 \right) \\
& \phantom{=} + ( 1  - h_1)
\left[\widetilde \Lambda_0 + \frac{1}{d+1} \sum_{n=1}^\infty (1 - h_2) \ldots (1 - h_n) h_{n+1} \v_{\xi_{n+1}}  \right]. 
\end{split}
\]
Notice that the infinite sum in brackets is equal in distribution with $\widetilde\Lambda$ and it is independent of
$h_1$ and $\xi_1$. Since $\frac{1}{d+1} \v_i + \widetilde\Lambda_0 = u_i$, where $(u_i)_{i=0,\ldots, d}$ are the usual
unit vectors in $\R^{d+1}$, we obtain the distributional equality
\begin{equation} \label{eq:d-perpet}
\widetilde\Lambda \stackrel{\mathcal{D}}{=}
h u_{\xi} + (1-h)\widetilde \Lambda,
\end{equation}
where on the right-hand side $\xi, h, \widetilde \Lambda$ are independent.
Applying now Theorem 1.1 in \cite{HL} 
(or the results in the proof of Theorem 3.4 in \cite{Seth})
with $Y = h \sim \mathrm{beta}(1,d)$,
and $B = u_\xi \sim \sum_{i=0}^d \frac{1}{d+1} \delta_{u_i}$, 
we obtain obtain the theorem. 
\end{proof}

\section{Regular polygons with an odd number of vertices} \label{sect:5}

Let $k$ be an odd positive integer, and assume $k\geq 5$. Let $K$ be a regular $k$-gon
with circumradius $1$, centroid $(0,0)$,
such that $(0,1)$ is a vertex and the side $v_1 v_2$ is parallel to the $x$-axis. We denote the vectors pointing from the origin to the vertices of $K$
in the counterclockwise order by $v_1, \ldots, v_k$. (To avoid confusion, we distinguish between points and vectors.) Put $K_0 = K$, and consider the process as before. For
simplicity we usually omit $k$ from our notation, and assume that $k$ is fixed, odd, and clear from the circumstances.

Obviously, $K_n$ is a polygon for each $n$, and since it is the intersection of translated copies of
$K$, its sides are parallel to the sides of $K$. However, note that $K_n$ is not necessarily a $k$-gon.
For convenience, we are still going to consider $K_n$ as a (possibly degenerated) $k$-gon with the following definitions.
Let $\ell_i$ and $\ell_i'$ be two parallel support lines of $K_n$ with equations $\ell_i \colon \langle x \, , \,v_i \rangle=\alpha_i$
and  $\ell_i' \colon \langle x \, , \,v_i \rangle=\alpha_i'$, where $\alpha_i>\alpha_i'$. Now, we denote
$K_n\cap \ell_i$ by $A_i=A_i(n)$ and we consider it as the $i$th vertex of $K_n$. Similarly, $K_n\cap \ell_i'$ is
denoted by $s_i=s_i(n)$ and we call it the $i$th side of $K_n$. Note that with these notation some vertices might
coincide and correspondingly some sides might degenerate into a point. We also introduce the $i$th height of 
$K_n$ as $m_i(n)=\alpha_i-\alpha_i'$. We put $\m_n=(m_1(n), m_2(n), \ldots, m_k(n))$, and $m_n=\max_i m_i(n)$.

The radius of the inscribed circle of $K$ is denoted by $\rho_k= \cos (\pi/k)$. We also introduce the notion of
change region here:
\[
\mathcal R_i(n)=K_n\cap \{ x\; | \; \langle x \, , \, v_i \rangle\geq \alpha_i'+\rho_k \}, \quad i=1,2,\ldots, k.
\]
Intuitively, the $i$th side moves, if we choose the next random point in $\mathcal R_i$.
(Note that, this is not entirely true, since a degenerated side can move in other ways.) Obviously, 
if $p_{n+1}\notin \bigcup_1^k \mathcal R_i(n)$, then $K_{n+1}=K_n$. 

We define $$K_\infty= \bigcap_{n=0}^\infty K_n,$$  the so called {\em limit object}.

\begin{lemma} \label{lemma:limit-obj}
The limit object $K_\infty$ is a possibly degenerated, closed $k$-gon whose sides are parallel to the sides of $K$. Furthermore,
the maximal height of $K_\infty$ is exactly $\rho_k$ almost surely.
\end{lemma}

\begin{proof}
Since $K_\infty$ is the intersection of closed half-planes with possible outer normals $-v_1, \ldots, -v_k$, it
follows, that $K_\infty$ is a closed, possibly degenerated $k$-gon with sides parallel to the sides of $K$.

First we show that  no height of $K_\infty$ is larger than $\rho_k$. Suppose that $m_1(\infty)>\rho_k$, in 
this case $\mathcal R_1(\infty)$ is of positive area. Observe that no point was selected from $\mathcal R_1(\infty)$
by definition, which is a contradiction.

Next we prove that the maximal height of $K_\infty$ is at least $\rho_k$. Clearly, it is enough to see 
that  $m_n\geq \rho_k$ for every $n$. This follows from the observation that  if $p_{n+1}\notin \bigcup_1^k \mathcal R_i(n)$,
then $K_{n+1}=K_n$. 
\end{proof}

In the following lemma we show that $K_n$ always contains a small circle of radius $1/10$. In particular this implies
that the area of $K_n$ (and thus the area of $K_\infty$ as well) is uniformly bounded from below by $\pi/100$. To ease the notation we
put $p_0 = 0$.

\begin{lemma}\label{termin}
Let $k\geq 5$, and assume that $$K_n=\bigcap_{j=0}^{n} (K+p_j),$$ where $p_j\in \bigcap_{m=0}^{j-1} (K+p_m)$ for all $j$.
Then $K_n$ contains a circle of radius $1/10$.
\end{lemma}

\begin{proof}
Denote $B$ the unit circle centered at the origin, which is the circumcircle of $K$ by definition. 
Also by definition $\rho_k B$ is the incircle of $K$. We consider $$B_n=\bigcap_{j=0}^{n} (B+p_j),$$ and we observe
that $K_n\subset B_n$ holds for all $n$. 

We claim that for all $j=0,1,\ldots, n$, we have $p_j\in B_n$. By definition $p_j \in K_j\subset B_j$.
Suppose that $p_j \notin B_n$, then there exist an index $n_0$ with $j<n_0\leq n$ such that $p_j \notin (B+p_{n_0})$,
and thus $p_{n_0} \notin (B+p_j)$. But by definition $p_{n_0}\in B_{n_0}\subset (B+p_j)$, a contradiction.

We obtained that $B_n$ is the intersection of the unit circles $B+p_j$ such that all centers $p_j$ are contained
in $B_n$. This readily implies that the minimal width of $B_n$ is at least one. Then Blaschke's Theorem (see \cite{jb}, p. 18, Th. 2-5.)
implies that there exists $x$ such that $B/3+x\subseteq B_n$. Obviously for all $j\leq n$ we have that $x \in 2B/3+p_j$,
and thus $\rho_k\geq \rho_5=\cos \pi/5 \approx 0.809>2/3+1/10$ implies that for all $j\leq n$ we have $B/10+x\subset K+p_j$,
which proves the statement. %(Remark: the proof works for any integer $k\geq 6$.)
\end{proof}

%The case $k=5$ is settled in the next section in Lemma~\ref{pentabeirtkor}.

\begin{lemma}
There exists a $\delta_k>0$ such that if every height of $K_n$ is smaller than $\rho_k+\delta_k$, then the 
change regions $\mathcal R_i$ are pairwise disjoint.
\end{lemma}

\begin{proof}
We show that for every $i\neq j$ $\mathcal R_i$ and $\mathcal R_j$ are disjoint. % We are going to prove by induction on $m=\min(|i-j|-1, k-1-|i-j|)$.

First we show that the statement is true for adjacent regions. Suppose that 
$X\in \mathcal R_1 \cap \mathcal R_2$ (see Figure  \ref{fig:consecutivedr}). 

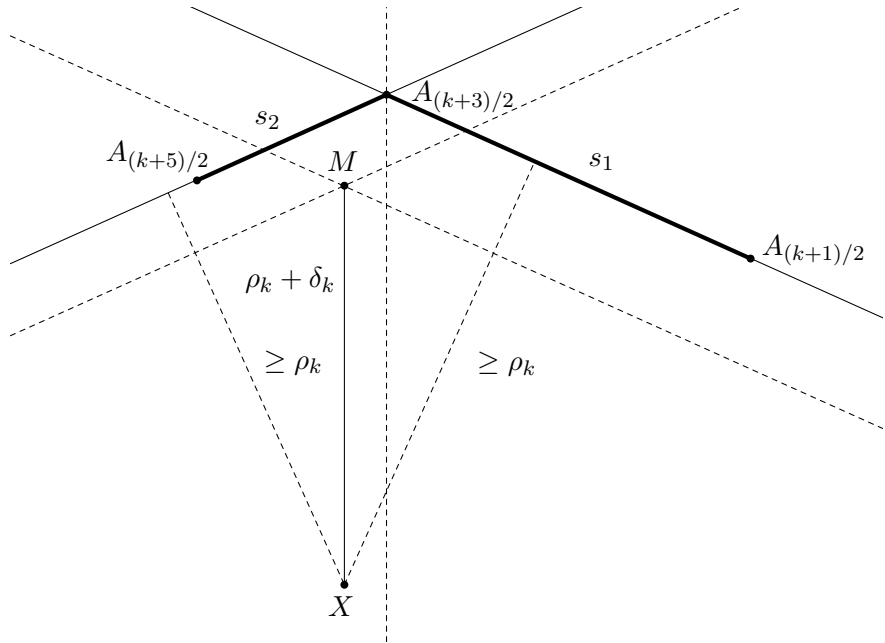
\begin{figure}
\centering
\begin{tikzpicture}[line cap=round,line join=round,>=triangle 45,x=1.0cm,y=1.0cm]
\clip(5.38,-3.74) rectangle (16.98,4.7);
\draw [dash pattern=on 2pt off 2pt] (10.38,-3.74) -- (10.38,4.7);
\draw [domain=5.38:16.98] plot(\x,{(-3.81--1.52*\x)/3.38});
\draw [domain=5.38:16.98] plot(\x,{(--27.74-1.52*\x)/3.38});
\draw [line width=1.6pt] (7.86,2.4)-- (10.38,3.54);
\draw [line width=1.6pt] (10.38,3.54)-- (15.22,1.36);
\draw [dash pattern=on 2pt off 2pt] (7.48,2.23)-- (9.82,-2.98);
\draw [dash pattern=on 2pt off 2pt] (9.82,-2.98)-- (12.35,2.65);
\draw [dash pattern=on 2pt off 2pt,domain=5.38:16.98] plot(\x,{(-7.05--1.52*\x)/3.38});
\draw [dash pattern=on 2pt off 2pt,domain=5.38:16.98] plot(\x,{(--32.65-2.18*\x)/4.84});
\draw (9.82,2.33)-- (9.82,-2.98);
\fill [color=black] (10.38,3.54) circle (1.5pt);
\draw[color=black] (11.4,3.51) node {$A_{(k+3)/2}$};
\fill [color=black] (7.86,2.4) circle (1.5pt);
\draw[color=black] (7.34,2.71) node {$A_{(k+5)/2}$};
\fill [color=black] (15.22,1.36) circle (1.5pt);
\draw[color=black] (16.06,1.47) node {$A_{(k+1)/2}$};
\draw[color=black] (8.8,3.2) node {$s_2$};
\draw[color=black] (13.24,2.65) node {$s_1$};
\fill [color=black] (9.82,-2.98) circle (1.5pt);
\draw[color=black] (9.78,-3.26) node {$X$};
\draw[color=black] (9.14,-0.02) node {$\geq \rho_k$};
\draw[color=black] (11.98,-0.04) node {$\geq \rho_k$};
\fill [color=black] (9.82,2.33) circle (1.5pt);
\draw[color=black] (9.8,2.66) node {$M$};
\draw[color=black] (9.1,1.08) node {$\rho_k+\delta_{k}$};
\end{tikzpicture}
\caption{ \label{fig:consecutivedr}
 Adjacent change regions are disjoint}
\end{figure}

According to Figure  \ref{fig:consecutivedr} we draw two lines parallel to $\ell'_1$ and $\ell'_2$
respectively that are at distance exactly $\rho_k$ from the point $X$, these two lines meet in the point $M$. 
Obviously, there exists a $\delta_{k}>0$ (depending only on $k$), such that $\overline{XM}=\rho_k+\delta_{k}$. 
Readily follows that $m_{(k+3)/2}\geq \rho_k + \delta_{k}$, a contradiction.

Next we prove that if $2\leq m\leq (k-1)/2$, and $X\in \mathcal R_1\cap\mathcal R_m$, then 
$X\in \bigcap_1^m \mathcal R_j$. This obviously implies the statement of the lemma. We proceed by induction
on $m$. For $m=2$ we are done. Now we assume that the statement is true till $m-1$, and we prove it for $m$.

Pick $X\in \mathcal R_1 \cap \mathcal R_{m}$. We may assume that $X \notin \mathcal R_j$ for any $j=2, 3, \ldots, m-1$,
otherwise we would be done by applying the hypothesis twice. We may also assume that we changed the coordinate
system such that the slope of $\ell'_m$ is positive, the slope of $\ell'_1$ is negative, and the bisectors of the
line $\ell'_1$ and $\ell'_m$ are vertical and horizontal, see Figure \ref{fig:nonconsecutivedr}.

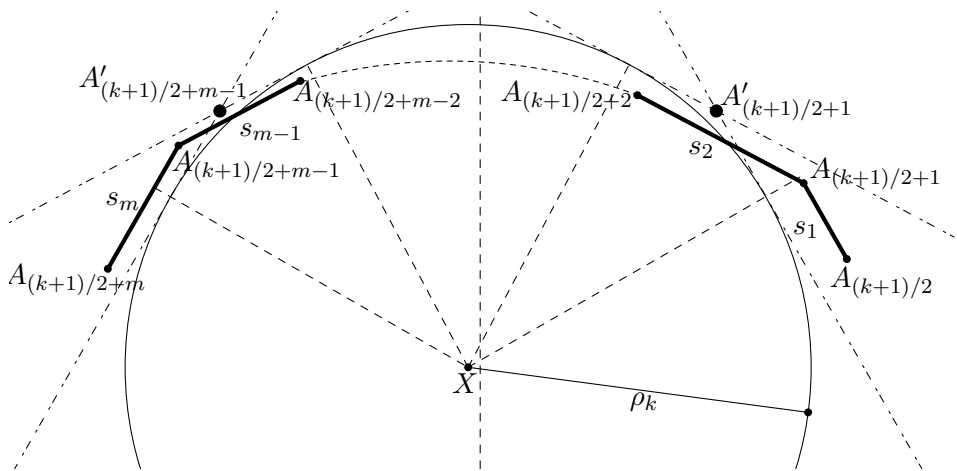
\begin{figure}
\centering

\begin{tikzpicture}[line cap=round,line join=round,>=triangle 45,x=.5cm,y=.5cm]
\clip(-2.42,-6.84) rectangle (22.7,5.32);
\draw [dash pattern=on 3pt off 3pt] (10.1,-6.84) -- (10.1,5.32);
\draw [line width=1.6pt] (0.2,-1.52)-- (2.08,1.76);
\draw [line width=1.6pt] (2.08,1.76)-- (5.32,3.48);
\draw [line width=1.6pt] (18.7,0.76)-- (19.85,-1.26);
\draw [line width=1.6pt] (14.28,3.1)-- (18.7,0.76);
\draw [shift={(9.28,-10.24)},dash pattern=on 2pt off 2pt]  plot[domain=1.21:1.85,variable=\t]({1*14.25*cos(\t r)+0*14.25*sin(\t r)},{0*14.25*cos(\t r)+1*14.25*sin(\t r)});
\draw(9.78,-4.14) circle (4.56cm);
\draw [dash pattern=on 3pt off 3pt] (1.44,0.64)-- (9.78,-4.14);
\draw [dash pattern=on 3pt off 3pt] (9.78,-4.14)-- (5.5,3.91);
\draw [dash pattern=on 3pt off 3pt] (9.78,-4.14)-- (14.06,3.91);
\draw [dash pattern=on 3pt off 3pt] (9.78,-4.14)-- (18.6,0.92);
\draw [dash pattern=on 1pt off 2pt on 3pt off 4pt,domain=-2.42:22.7] plot(\x,{(-5.39--3.28*\x)/1.88});
\draw [dash pattern=on 1pt off 2pt on 3pt off 4pt,domain=-2.42:22.7] plot(\x,{(--3.22--1.72*\x)/3.24});
\draw [dash pattern=on 1pt off 2pt on 3pt off 4pt,domain=-2.42:22.7] plot(\x,{(-36.86--1.72*\x)/-3.24});
\draw [dash pattern=on 1pt off 2pt on 3pt off 4pt,domain=-2.42:22.7] plot(\x,{(--36.04-2.01*\x)/1.15});
\draw (9.78,-4.14)-- (18.82,-5.34);
\fill [color=black] (0.2,-1.52) circle (1.5pt);
\draw[color=black] (-0.65,-1.8) node {$A_{(k+1)/2+m}$};
\fill [color=black] (2.08,1.76) circle (1.5pt);
\draw[color=black] (4.15,1.2) node {$A_{(k+1)/2+m-1}$};
\fill [color=black] (5.32,3.48) circle (1.5pt);
\draw[color=black] (7.36,3.06) node {$A_{(k+1)/2+m-2}$};
\draw[color=black] (0.6,0.28) node {$s_{m}$};
\draw[color=black] (4.54,2.1) node {$s_{m-1}$};
\fill [color=black] (14.28,3.1) circle (1.5pt);
\draw[color=black] (12.4,3.06) node {$A_{(k+1)/2+2}$};
\fill [color=black] (18.7,0.76) circle (1.5pt);
\draw[color=black] (20.62,1) node {$A_{(k+1)/2+1}$};
\fill [color=black] (19.85,-1.26) circle (1.5pt);
\draw[color=black] (20.78,-1.88) node {$A_{(k+1)/2}$};
\draw[color=black] (18.76,-0.48) node {$s_1$};
\draw[color=black] (15.98,1.64) node {$s_2$};
\fill [color=black] (9.78,-4.14) circle (1.5pt);
\draw[color=black] (9.7,-4.55) node {$X$};
\fill [color=black] (18.82,-5.34) circle (1.5pt);
\draw[color=black] (14.46,-5.06) node {$\rho_k$};
\fill [color=black] (3.18,2.68) circle (2.5pt);
\draw[color=black] (1.68,3.4) node {$A'_{(k+1)/2+m-1}$};
\fill [color=black] (16.38,2.68) circle (2.5pt);
\draw[color=black] (18.24,2.94) node {$A'_{(k+1)/2+1}$};
\end{tikzpicture}
\caption{\label{fig:nonconsecutivedr}
 Non-adjacent change regions are disjoint}
 \end{figure}

Draw the translated copy $K_X$ of $K$ whose center is $X$, the incircle of $K_X$ is of radius $\rho_k$ and of
center $X$. Consider the vertices $A_{(k+1)/2+1}$ and $A_{(k+1)/2+m-1}$ of $K_n$, and the vertices $A'_{(k+1)/2+1}$ 
and $A'_{(k+1)/2+m-1}$ of $K_X$. From the assumptions it clearly follows that the `horizontal distance'
(the difference of the $x$ coordinates) of $A_{(k+1)/2+1}$ and $A_{(k+1)/2+m-1}$ is larger than the horizontal distance
of $A'_{(k+1)/2+1}$ and $A'_{(k+1)/2+m-1}$. But this is a contradiction, since the sides $s_1, s_2, \ldots, s_{m-1}$ form
a fixed angle with the $x$-axis, and each of them is at most as long as the side length of $K$, and thus the horizontal
distance of  $A'_{(k+1)/2+1}$ and $A'_{(k+1)/2+m-1}$ is maximal.
\end{proof}

A configuration is called \textit{reduced} if the change regions
are disjoint. In a reduced state it is possible to follow the process. That gives
the importance of the following simple corollary which readily follows from the fact that $\m_n$ is 
componentwise monotone decreasing and $m_n \downarrow \rho_k$.

\begin{cor} \label{cor:reduced}
The process a.s.~reaches  a
reduced state in a random number of steps. After reaching a reduced state, the process always stays in a reduced state.
\end{cor}

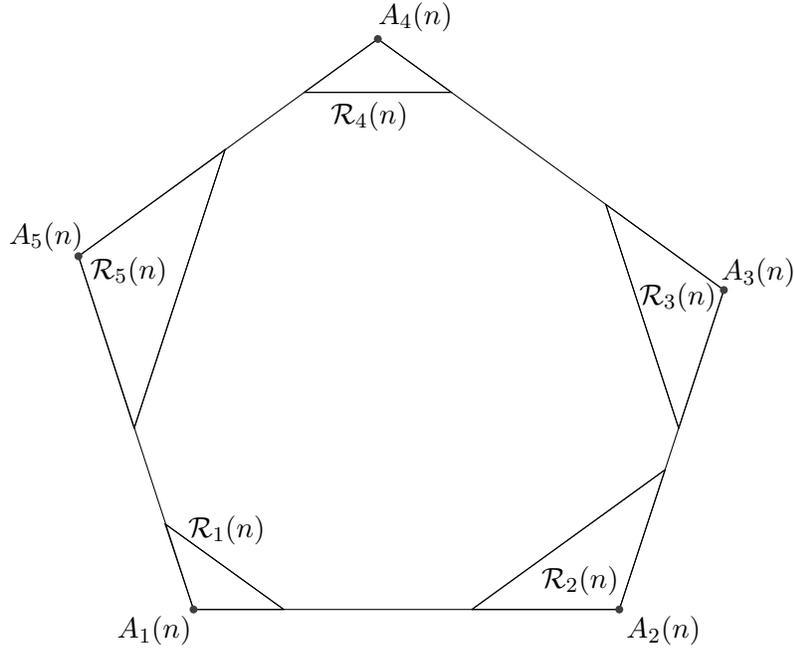
\begin{figure}
\definecolor{uququq}{rgb}{0.25,0.25,0.25}
\begin{tikzpicture}[line cap=round,line join=round,>=triangle 45,x=1.0cm,y=1.0cm]
\clip(-3.5,5) rectangle (8.89,14.18);
\draw (-0.14,5.84)-- (5.52,5.84);
\draw (5.52,5.84)-- (6.91,10.09);
\draw (6.91,10.09)-- (2.31,13.43);
\draw (2.31,13.43)-- (-1.67,10.54);
\draw (-1.67,10.54)-- (-0.14,5.84);
\draw (1.33,12.72)-- (3.29,12.72);
\draw (0.28,11.96)-- (-0.93,8.25);
\draw (5.34,11.23)-- (6.31,8.25);
\draw (6.13,7.7)-- (3.56,5.84);
\draw (1.06,5.84)-- (-0.52,6.98);
\draw (3.56,5.84)-- (5.52,5.84);
\draw (5.52,5.84)-- (6.13,7.7);
\draw (6.13,7.7)-- (3.56,5.84);
\draw (6.31,8.25)-- (6.91,10.09);
\draw (6.91,10.09)-- (5.34,11.23);
\draw (5.34,11.23)-- (6.31,8.25);
\draw (3.29,12.72)-- (2.31,13.43);
\draw (2.31,13.43)-- (1.33,12.72);
\draw (1.33,12.72)-- (3.29,12.72);
\draw (0.28,11.96)-- (-1.67,10.54);
\draw (-1.67,10.54)-- (-0.93,8.25);
\draw (-0.93,8.25)-- (0.28,11.96);
\draw (-0.14,5.84)-- (1.06,5.84);
\draw (1.06,5.84)-- (-0.52,6.98);
\draw (-0.52,6.98)-- (-0.14,5.84);
\draw (-0.66,5.56) node {$A_1(n)$};
\draw (0.3,6.9) node {$\mathcal{R}_1(n)$};
\draw (6.11,5.56) node {$A_2(n)$};
\draw (5.0,6.23) node {$\mathcal{R}_2(n)$};
\draw (7.37,10.32) node {$A_3(n)$};
\draw (6.3,10.0) node {$\mathcal{R}_3(n)$};
\draw (2.81,13.73) node {$A_4(n)$};
\draw (2.2,12.4) node {$\mathcal{R}_4(n)$};
\draw (-2.1,10.8) node {$A_5(n)$};
\draw (-1.0,10.34) node {$\mathcal{R}_5(n)$};
\fill [color=uququq] (-0.14,5.84) circle (1.5pt);
\fill [color=uququq] (-1.67,10.54) circle (1.5pt);
\fill [color=uququq] (5.52,5.84) circle (1.5pt);
\fill [color=uququq] (6.91,10.09) circle (1.5pt);
\fill [color=uququq] (2.31,13.43) circle (1.5pt);
\end{tikzpicture}
\caption{\label{fig:change}
 Change regions in a reduced state}
\end{figure}

\section{The pentagon}\label{sect:pentagon}

In this section we consider the pentagon process. This is the simplest case when
not only the position, but also the shape of the limit object is random.
We show that exactly one height
of the limit object is $\rho_5$, which allows us to determine the speed
of the process.

\subsection{On the limit pentagon}

First we prove that the process cannot degenerate in the following sense.

\begin{lemma}
$K_n$ is always a pentagon with equal inner angles.
\end{lemma}

\begin{proof}
The key observation is that the directions of the sides of $K_n$ are prescribed,
thus the only thing we have to show that a side cannot disappear.
Suppose the opposite, and seek a contradiction. Let $K_n$ be the first non-pentagonal
state, and first assume that it is a quadrilateral and the side $A_1 A_5$ disappears.
It easy to calculate the inner
angles of $K_n$, three of them equals the inner angle of a
regular pentagon, $3\pi/5$ (at vertices $A_2$, $A_3$ and $A_4$), while the fourth
one is $\pi/5$ (at the vertex $A_1$). Also note, that the side lengths of $K_n$
cannot exceed the side length of $K$. Thus $K_n$ is contained in a deltoid, see
Figure~\ref{deltoid}, where $s$ is the side length of $K$. This implies that the
heights $m_2$ and $m_4$ of $K_n$ are at most
$s\cdot \sin (\pi/5)=2\cdot \sin^2(\pi/5)\approx 0.69$. A simple argument
shows that we may assume that $A_4$ was a vertex of $K_{n-1}$, but $A_1$
and $A_2$ were not. 
This implies that the side $A_1 A_2$ comes from $K$ (more precisely,
$A_1 A_2 \subset p_n + \partial K$), and so $m_4 \geq \rho_5$.
But this is not possible, since $m_4<\rho_5$, a contradiction.
Similar argument settles the case when $K_n$ is a triangle.
\end{proof}

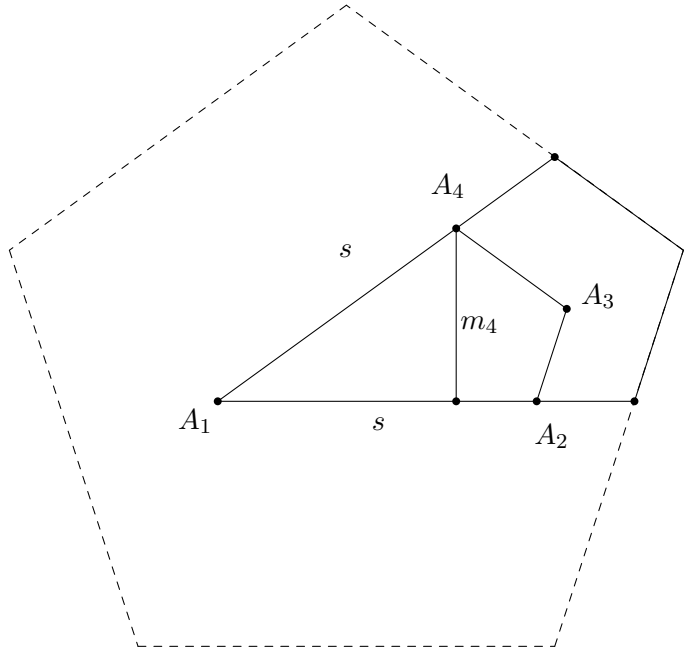
\begin{figure}[ht]
\centering
\begin{tikzpicture}[line cap=round,line join=round, >=triangle 45, x=1.0cm,y=1.0cm]
\clip(3.8,-5) rectangle (16.04,3.98);
\draw [dash pattern=on 3pt off 3pt] (9.61,3.69)-- (5.13,0.43);
\draw [dash pattern=on 3pt off 3pt] (5.13,0.43)-- (6.84,-4.84);
\draw [dash pattern=on 3pt off 3pt] (6.84,-4.84)-- (12.38,-4.84);
\draw [dash pattern=on 3pt off 3pt] (12.38,-4.84)-- (14.09,0.43);
\draw [dash pattern=on 3pt off 3pt] (14.09,0.43)-- (9.61,3.69);
\draw (7.9,-1.58)-- (13.44,-1.58);
\draw (12.38,1.67)-- (7.9,-1.58);
\draw (12.38,1.67)-- (14.09,0.43);
\draw (14.09,0.43)-- (13.44,-1.58);
\draw (11.07,0.72)-- (11.07,-1.58);
\draw (11.07,0.72)-- (12.54,-0.35);
\draw (12.54,-0.35)-- (12.14,-1.58);
\fill [color=black] (7.9,-1.58) circle (1.5pt);
\draw[color=black] (7.6,-1.84) node {$A_1$};
\fill [color=black] (13.44,-1.58) circle (1.5pt);
\fill [color=black] (12.38,1.67) circle (1.5pt);
\draw[color=black] (10.04,-1.88) node {$s$};
\draw[color=black] (9.6,0.42) node {$s$};
\fill [color=black] (12.14,-1.58) circle (1.5pt);
\draw[color=black] (12.34,-2.04) node {$A_2$};
\fill [color=black] (11.07,0.72) circle (1.5pt);
\draw[color=black] (10.96,1.28) node {$A_4$};
\fill [color=black] (12.54,-0.35) circle (1.5pt);
\draw[color=black] (12.96,-0.18) node {$A_3$};
\fill [color=black] (11.07,-1.58) circle (1.5pt);
\draw[color=black] (11.38,-0.56) node {$m_4$};
\end{tikzpicture}
\caption{The deltoid containing $K_n$}
\label{deltoid}
\end{figure}

By Corollary \ref{cor:reduced} in a random number of steps we reach a
reduced state, and so as in the simplex case we may and do assume that
the process starts from a reduced state. It also follows that in a reduced
state the change regions are always triangles.

Note that if the random point falls in $\mathcal{R}_1$ then beside $m_1$,
the opposite heights $m_3$ and $m_4$ also decrease. Some calculation shows
that if $m_1$ decreases by $x$ then $m_3$ and $m_4$ both decreas by
$c \, x$, with
\begin{equation} \label{c}
c= \frac{\sqrt{5} - 1}{2}
\end{equation}
being the ratio of the golden section.
We say that $m_i$ and $m_j$ are \textit{competing} heights, if
$m_i > \rho_5$, $m_j > \rho_5$, and they are not adjacent.

 To describe the dynamics of the process we define the following vectors:
$\v_1 = (1,0,c,c,0)$, $\v_2=(0,1,0,c,c)$,
$\v_3=(c,0,1,0,c)$, $\v_4=(c,c,0,1,0)$, and $\v_5=(0,c,c,0,1)$. With this notation,
if in a reduced state in the $(n+1)$th step the random point falls in $\mathcal{R}_i(n)$, then
\begin{equation} \label{eq:5szogmag}
\m_{n+1} = \m_n - h_n (m_i(n) - \rho_5) \v_i,
\end{equation}
where $h_1, h_2, \ldots$ are iid with common distribution function $H_2$ in (\ref{H_d}),
i.e.~$h$ is the distribution of the distance from the base of a uniformly
chosen point in a triangle with height 1. 
That is, $h_{n+1} (m_i(n) - \rho_5)$ is the distance of $p_{n+1}$ and the side
of $\mathcal{R}_i(n)$ which is opposite to $A_i(n)$. The probability of this event
is $|\mathcal{R}_i(n)|/|K_n|$, where $|\cdot|$ is the area.

\begin{lemma}\label{nincsnemszomszedos}
The limit pentagon cannot have non-adjacent heights equal to $\rho_5$.
\end{lemma}

\begin{proof}
Emphasizing that the process can be at any reduced state we omit the index $n$.

Assume that there is a state with at least 2 competing heights $> \rho_5$.
Let, say, $m_1$ be the maximum height, which has a competing pair, say  $m_3$.
If the maximum height has no competing pair $> \rho_5$ than its change has no
affect on the two competing heights. Thus $m_1$ will change eventually. So we
may and do assume that $m_1$ is the largest height.

\textbf{Case 1}: $c (m_1 - \rho_5) / 2 > m_3 - \rho_5$, with $c$ defined in (\ref{c}). Then the probability
that in the next change step the uniform random point falls in $\mathcal{R}_1$ 
is $> 1/5$, and given this the probability that $m_1$ decrease at least
with $(m_1 - \rho_5) /2$ equals $\p \{ h > 1/2 \} = 1/4$. In this case
$m_3$ decreases below $\rho_5$, and so the probability of this $\geq 1/20$.

\textbf{Case 2}: $c (m_1 - \rho_5) / 2 \leq m_3 - \rho_5$. The probability that
in the next change step the random point falls in $\mathcal{R}_3$ is
\[
\frac{(m_3 - \rho_5)^2}{\sum_{i=1}^5 (m_i - \rho_5)_+^2} \geq
\frac{(m_3 - \rho_5)^2}{5 (m_1 - \rho_5)^2} \geq \frac{c^2}{20}.
\]
We show that
with positive probability we end up in a state corresponding to Case 1. In the next step
\begin{align*}
m_1^\prime & = m_1 - c h (m_3 - \rho_5), \\
m_3^\prime & = m_3 - h (m_3 - \rho_5).
\end{align*}
We want an $h \in (0,1)$, such that
$c (m^\prime_1 - \rho_5) / 2 > m^\prime_3 - \rho_5$.
Some calculation shows that this happens if and only if
$$
h > \frac{1}{1 - \frac{c^2}{2}}
\left( 1 - \frac{c}{2} \frac{m_1 - \rho_5}{m_3 - \rho_5} \right),
$$
where the right side is 
$$
\leq \frac{1 - \frac{c}{2}}{1 - \frac{c^2}{2}} = \frac{3 \sqrt{5} - 5}{2}.
$$
The probability of this event is at least
$$
\p \left\{ h >  \frac{3 \sqrt{5} - 5}{2} \right\} = \frac{(7 - 3 \sqrt{5})^2}{4}
\approx 0.0213. 
$$
So we are almost in Case 1, but it can happen that $m_1^\prime$ is not maximal.
Notice that
$$
\frac{m_1^\prime - \rho_5}{m_1 - \rho_5}
= \frac{m_1 - \rho_5 - c (m_3 - \rho_5) h}{m_1 - \rho_5} \geq 1 - c,
$$
which implies that the probability of choosing in $\mathcal{R}_1$ in the next
change step is $\geq (1 - c)^2/5$.

So we showed that starting from any state with at least two competing
heights $> \rho_5$,
the probability that in two change steps one of them decreases below $\rho_5$ is
$$
\geq \frac{c^2}{20} \frac{(7- 3 \sqrt{5})^2}{4} \frac{(1-c)^2}{20}
\approx 2.97 \cdot 10^{-6}. 
$$ 
This proves that the process cannot have this configuration
for infinite number of steps.
\end{proof}

\begin{lemma}\label{nincsketszomszedos}
There is no non-regular pentagon with equal angles, in which the two largest heights
are consecutive.
\end{lemma}

\begin{proof}
As a first step we prove a somewhat surprising result that provides a linear relationship
between any four heights of the pentagon. We assume that $m_1, m_3$ and $m_4$ are given,
and we express $m_2$ as a linear combination of the previous three. To simplify the calculations, we place the pentagon
into a new coordinate system such that $A_1$ is the origin and $A_1A_2$ agrees with the
$x$-axis, and the whole pentagon lies in the upper half plane. Recall that $-v_1=(\cos(3\pi/10), \sin(3\pi/10)), \, -v_2=(\cos(7\pi/10),
\sin(7\pi/10)), \, -v_3=(\cos(11\pi/10), \sin(11\pi/10)), \,
-v_4=(0,-1)$, $-v_5=(\cos(-\pi/10), \sin(-\pi/10))$ are the outer normals of the sides, as we defined earlier. 
From the setup the equations of $\ell'_3=A_5A_1$ and $\ell'_4=A_1A_2$ readily follow:
$\ell'_3\colon \langle -v_3, (x,y) \rangle=0$ and $\ell'_4\colon y=0$. Using
the definition of $m_1$ we obtain $\ell'_1\colon \langle -v_1, (x,y)\rangle=m_1$.
And again by the definition of $m_3$ and $m_4$, $A_4$ is on the line of equation
$\ell_4\colon y=m_4$ and $A_3$ is on $\ell_3\colon \langle -v_3, (x,y) \rangle=-m_3$. We can express
$A_3$ and $A_4$ by solving the system of equations:
\begin{eqnarray*}
A_3 & = &
\left (\frac{m_1 \sin\frac{11\pi}{10}+m_3 \sin\frac{3\pi}{10}}{\sin \frac{8\pi}{10}}\; ,
\; \frac{m_1 \cos\frac{11\pi}{10}+m_3 \cos\frac{3\pi}{10}}{-\sin \frac{8\pi}{10}} \right ),\\
A_4  & = & \left(\frac{m_1-m_4\sin \frac{3\pi}{10}}{\cos\frac{3\pi}{10}} \; , \; m_4\right).\end{eqnarray*}
Now, we can find the equation of $\ell'_2$ and $\ell'_5$. After suitable simplifications,
introducing the golden ratio $\lambda=(\sqrt{5}+1)/2$, we obtain
\begin{eqnarray}\label{eqe2}
\ell'_2\colon \cos\frac{7\pi}{10} x+ \sin\frac{7\pi}{10} y=-m_1+\lambda m_4,\\
\label{eqe5} \ell'_5\colon \cos\frac{-\pi}{10} x+ \sin\frac{-\pi}{10} y=-m_1+\lambda m_3.
\end{eqnarray}
Thus $A_2=((-m_1+\lambda m_3)/\cos(-\pi/10),0)$, and to obtain $m_2$ we need to
calculate the distance between $A_2$ and $e_2$:
\begin{eqnarray*}
m_2 & = & \left|\cos\frac{7\pi}{10}\cdot \frac{-m_1+\lambda m_3}{\cos(\pi/10)}+
m_1-\lambda m_4 \right|= 
\left |\left(\frac1\lambda+1\right)m_1-m_3- \lambda m_4 \right | \\
& = &  | \lambda m_1 - m_3 - \lambda m_4|.
\end{eqnarray*}
From (\ref{eqe2}) and (\ref{eqe5}) it readily follows that $m_3>m_1/\lambda$
and $\lambda m_4> m_1$, hence
\begin{equation}\label{aranymetszes}
m_2=-\lambda m_1 + m_3 + \lambda m_4.
\end{equation} 
Now, suppose that $m_1$ and $m_2$ are the two largest heights. If $m_2\neq m_3$,
then we have a contradiction by (\ref{aranymetszes}). If $m_2=m_3$, then since
$m_1$ and $m_2$ are the two largest, it follows that $m_1=m_2=m_3=m_4$, and
hence the pentagon is regular.
\end{proof}

As a consequence of the previous lemmas we obtain

\begin{theorem} \label{thm:limit-pentagon}
The limit pentagon has exactly one height equal to $\rho_5$ a.s.
\end{theorem}

{\it Remark.} With a rather tedious case analysis one can prove that for any height of the limit pentagon $m_i \geq \rho_5+2-4c\approx 0.33688$, which is sharp. 

\subsection{Rapidness of the pentagon process}

In the previous section we proved that the limit pentagon has exactly one height equal
to $\rho_5$ a.s., i.e.~after finite number of steps $K_n$ has only one height
greater than $\rho_5$. This observation allows us to prove some asymptotic
results for the speed, however, as the area of the limit is now \textit{random},
we cannot prove limit theorem, only upper and lower bounds.

Let denote $t^*$ the maximum and $t_*$ the minimum of the area of the possible limit pentagons.
Note that $t_* \geq \pi/100$ by Lemma \ref{termin}. Then we have the following.

\begin{theorem} \label{th:pentagon-rapid}
For any $x > 0$
\[
\begin{split}
\e^{-\frac{x^2}{t_*}} & \leq
\liminf_{n \to \infty}
\p \left\{ \sqrt{ n \tan \frac{3\pi}{10} } (m_n -\rho_5 ) > x  \right\} \\
& \leq 
\limsup_{n \to \infty}
\p \left\{ \sqrt{ n \tan \frac{3\pi}{10} } (m_n -\rho_5 ) > x  \right\} \leq
\e^{-\frac{x^2}{t^*}}.
\end{split}
\]
Moreover,
\[
\E \sqrt{ n \tan \frac{3\pi}{10} } (m_n - \rho_5 ) =
\frac{\E \sqrt{t} }{4 \sqrt{\pi}},
\]
where $t$ denotes the area of the limit pentagon.
\end{theorem}

\begin{proof}
Put $t_n = |K_n|$. 
Once there is only one height $> \rho_5$ the limit pentagon is determined and
so is its area $\lim_{n \to \infty} t_n = t$. The area of the only non-empty change region
$| \mathcal{R}_i(n)| = (m_n - \rho_5)^2 \tan \frac{3 \pi}{10}$. This means that the
height process $\ell_n = m_n - \rho_5$ behaves as
\[
\ell_{n+1} = 
\begin{cases}
\ell_n ( 1 - h_{n+1} ), &\textrm{w.p.} \ \frac{\ell_n^2}{t_n} \tan \frac{3 \pi}{10} , \\
\ell_n, & \textrm{w.p.} \ 1-  \frac{\ell_n^2}{t_n} \tan \frac{3 \pi}{10},
\end{cases}
\]
where $h, h_1, h_2, \ldots$ are iid, $\p \{ 1- h \leq x \} = x^2$, for $x \in [0,1]$.
Since $t_n \downarrow t$ a.s., by Lemma \ref{lemma:recursion}
with $\delta=2$, $a_n = t_n$, $c = \tan (3\pi/10)$
we obtain that given $t$ we have for any $x > 0$
\[
\p \left\{ \sqrt{ \frac{n \tan \frac{3\pi}{10}}{t} } (m_n -\rho_5 ) > x \big| t \right\}
\to \e^{-x^2},
\]
or
\[
\p \left\{ \sqrt{ n \tan \frac{3\pi}{10} } (m_n -\rho_5 ) > x \big| t \right\}
\to \e^{-\frac{x^2}{t}}.
\]
The convergence of the moments also hold (as in Lemma~\ref{lemma:recursion}), in particular
\[
\begin{split}
\E \sqrt{ n \tan \frac{3\pi}{10} } (m_n - \rho_5 ) & =
\E \left[ \E \left[ \sqrt{ n \tan \frac{3\pi}{10} } (m_n- \rho_5 ) | t \right] \right] \\
& \to \E \int_0^\infty \e^{-\frac{x^2}{t}} \d x =
\frac{\E \sqrt{t} }{4 \sqrt{\pi}},
\end{split}
\]
and the theorem is proved.
\end{proof}

\section{Rapidness estimates} \label{sect:sok}

In general the polygon process is too complicated to say anything more about the limit object
than Lemma \ref{lemma:limit-obj}. According to this lemma the maximal height of the limit object is $\rho_k$.
Using stochastic majorization and minorization we are able to determine the order of the convergence.

\begin{theorem}\label{mainodd}
For any $x > 0$ we have
\[
0 < \liminf_{n \to \infty} \p \{ \sqrt{n} (m_n - \rho_k) > x \}
\leq \limsup_{n \to \infty} \p \{ \sqrt{n} (m_n - \rho_k) > x \} < 1.
\]
\end{theorem}

\begin{figure} 
\begin{center}
\includegraphics[scale=0.25]{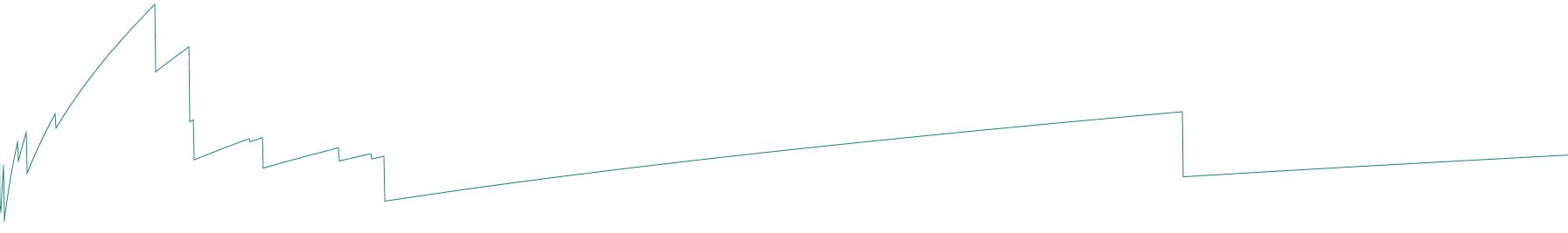} 
\caption{\label{fig:ff} The evolution of $\sqrt n (m_n - \rho_7)$ for 1800 iterations of the heptagon process}
\end{center}
\end{figure}

\begin{proof}
Let $\m_n = (m_1(n), \ldots, m_k(n))$ be the height vector, $m_n$ its maximum, and
$A_n = \sum_{i=1}^k |\mathcal{R}_i(n)|$ the area of the change regions.
By Corollary \ref{cor:reduced} we may and do assume that the change regions are already disjoint.
The probability of no change is the probability that the random point does
not fall in $\bigcup \mathcal{R}_i(n)$, is
$\p \{ \m_{n+1} = \m_n \} = 1 - A_n / |K_n|$.
The probability of change is $ A_n / |K_n|$, in particular
$|\mathcal{R}_i(n)| / |K_n|$ is the probability that we choose the point in
$\mathcal{R}_i(n)$. In this case $m_i (n+1) = m_i(n) - h_{n+1}^i (m_i(n) - \rho)$,
and all the other heights decrease at most with $ h_{n+1}^i (m_i(n) - \rho)$,
where $h_{n+1}^i (m_i(n) - \rho)$ is the distance from the base of a uniformly chosen point
in $\mathcal{R}_i(n)$, and so $h_{n+1}^i$ is the distance from the base of a uniformly
chosen point in $ \mathcal{R}_i(n) (m_i(n) - \rho)^{-1}$, i.e.~we scale the change region
to have height 1. So we have that in case of change
$\m_{n+1} \geq \m_n - h_{n+1}^i (m_n - \rho) \unit$, where $\unit$ stands for the constant 1 vector, and so
$m_{n+1} \geq m_n - h_{n+1}^i (m_n - \rho)$.
\smallskip

We want to construct simple processes, serving as lower and upper bound for $m_n$.
In order to do so we recall same basic properties of stochastic ordering. For random
variables $X$ and $Y$ we say that $X$ is \textit{stochastically larger} than $Y$
($Y \stleq X$) if $\p \{ X \leq x \} \leq \p \{ Y \leq x \}$ for any $x \in \R$. This is equivalent
to the condition $\E f(X) \geq \E f(Y)$ for any increasing function $f$. For random vectors the 
definition is somewhat trickier. In $\R^k$ a set $U$ is an \textit{upper set} if 
for $\x_1 \in U$, $\x_2 \geq \x_1$ imply $\x_2 \in U$. For $k$-dimensional random vectors $\X$  
and $\Y$ we have $\Y \stleq \X$ if $\p \{ \X \in U \} \geq \p \{ \Y \in U \}$ for any upper set $U$.
This is equivalent to the condition $\E f(\X) \geq \E f(\Y)$ for any componentwise increasing function
$f: \R^k \to \R$. We refer to Shaked and Shanthikumar \cite{SS} chapter 1.A and chapters 6.A and 6.B.

The first step is to obtain a stochastic majorant and minorant
for $h_n^i$ for any type of scaled change regions. 
Let us fix such a region, and let $t_x$ be the area of those points in the region, which
are farther than $1 - x$ from the base. If $h$ is the distance of the random point from
the base then
$\p \{ h > 1 - x \} = t_x / t_1$. The angle of the upper vertex is
$\leq \frac{k-2}{ k } \pi$, and the corresponding angle bisector is orthogonal
to the base, so for all $x \in [0,1]$
$$
t_x \leq \frac{1}{2} x \, 2x \tan \frac{(k - 2) \pi}{2k}
= x^2  \tan \frac{(k - 2) \pi}{2k}.
$$
By Lemma \ref{termin} a disc of radius $1/10$ is contained in $K_n$, which together with
convexity imply that the angle of the upper vertex is $\geq 2 \arcsin \frac{1}{20}$.
Therefore
$$
t_x \geq x^2 \tan \left( \arcsin \frac{1}{20} \right) =: x^2 \delta_1.
$$
Summarizing, we have
$$
x^2 \delta_1  \leq \p \{ h > 1 - x \} = \frac{t_x}{t_1} \leq x^2 c_1 , 
$$
where $c_1 = \tan \frac{(k - 2) \pi}{2 k} >>1$. Note that $\delta_1$ in the lower
bound does not depend on $k$. For $x \geq 0$ put
\begin{eqnarray} \label{h-major}
H^*(x) & =  & \min \{ x^2 c_1, 1\}, \\
\label{h-minor}
H_*(x) & = & \left\{
\begin{array}{ll}
x^2 \delta_1, & x \in [0,1), \\
1, & x \geq 1,
\end{array}
\right.
\end{eqnarray}
for the corresponding distribution functions of $1 - h$.

The previous reasoning also shows that
$$
\delta_1 (m_i(n) - \rho_k)_+^2 \leq |\mathcal{R}_i(n) | \leq c_1 (m_i(n) - \rho_k)_+^2,
$$
and so
\begin{equation} \label{eq:A_n-bound}
\delta_1 (m_n - \rho_k)^2 \leq A_n  \leq k c_1 (m_n - \rho_k)^2.
\end{equation}
By the trivial bound and by Lemma~\ref{termin} we have the following upper and lower bounds for the area:
\begin{equation} \label{eq:ter-est}
\pi/ 100 \leq |K_n| \leq |K| \leq \pi.
\end{equation}

\textbf{The lower bound.}
Using (\ref{eq:A_n-bound}) and (\ref{eq:ter-est}) the change probability can be estimated as
$$
\frac{A_n}{|K_n|} \leq \frac{100 k c_1}{\pi} (m_n - \rho_k)^2 =:
c_2 (m_n - \rho_k)^2.
$$
Let us define the process
\begin{equation} \label{m-prime}
m_{n+1}' = \left\{
\begin{array}{ll}
m_n' - (m_n' - \rho) h_{n+1} & \textrm{w.p. } c_2  ( m_n' - \rho_k)^2, \\
m_n', & \textrm{w.p. } 1 - c_2  (  m_n' - \rho_k)^2,  \\
\end{array} \right.
\end{equation}
where $h, h_1, h_2, \ldots, $ iid, and $1 - h$ has distribution function $H^*$ in (\ref{h-major}).
We claim that
\begin{equation} \label{1-step-dom}
\p \{ m_{n+ 1} \leq x | m_n = y \} \leq 
\p \{ m_{n+ 1}' \leq x | m_n' = y \}.
\end{equation}
Indeed, $m_n'$ decreases with higher probability, and if it decreases, decreases larger.
Putting $ \l_n' = m_n' - \rho_k$
\[
\l_{n+1}' = \left\{
\begin{array}{ll}
\l_n' (1 - h_{n+1})  & \textrm{w.p. }   c_2 ( \l_n' )^2, \\
\l_n', & \textrm{w.p. } 1 - c_2 ( \l_n' )^2.
\end{array} \right.
\]
We can write $\l_{n+1}' = \min\{ \l_n', U_{n+1} \}$,
with $U_{n+1}$ independent from $\l_n'$ and having distribution function
\[
\p \{ U_{n+1} \leq x \} =
\left\{
\begin{array}{ll}
c_1 c_2 x^2, &  x < \frac{\l_n'}{\sqrt{c_1}}, \\
c_2 \, (\l_n')^2, & x \in [ \frac{\l_n'}{\sqrt{c_1}}, \l_n' ], \\
\textrm{anything}, & x > \l_n'.
\end{array} \right.
\]
If $V$ has distribution function
\begin{equation} \label{h-tilde}
\widetilde H(x) = \min \left\{ c_1 \, c_2 x^2, 1 \right\},
\end{equation}
then $\min\{  \ell_n', U_{n+1} \} \stgeq \min\{   \ell_n', V \}$
for any $n$ and $  \l_n'$.
For $V_1, V_2, \ldots$ iid with distribution function $\widetilde H$, put
$\underline V_n = \min \{ V_1, \ldots, V_n \}$. We obtained that for all $n$
$$
\p \{    \ell_{n+ 1}' \leq x |   \ell_n' = y \}
\leq \p \{ \underline V_{n+ 1} \leq x | \underline V_n = y \},
$$
combining this with (\ref{1-step-dom}) we deduce
\begin{equation} \label{1-step-dom-V}
\p \{ \l_{n+ 1} \leq x | \l_n = y \}
\leq \p \{ \underline V_{n+ 1} \leq x | \underline V_n = y \}.
\end{equation}
We claim that these inequalities imply the unconditional inequality.

The latter process can be written as (we assume that the process starts from
a sufficiently small state)
\[
\underline V_{n+1} = \left\{
\begin{array}{ll}
\underline V_n k_{n+1}, & \textrm{w.p. } c_1 c_2 \underline V_n^2, \\
\underline V_n, & \textrm{w.p. } 1 - c_1 c_2 \underline V_n^2,
\end{array}
\right.
\]
where $k_1, k_2, \ldots $ are iid with distribution function $\p \{ k \leq x \} = x^2$, $x \in [0,1]$.
Short calculation gives that
$$
\p \left\{ \underline V_{n+1} \leq x | \underline V_n = y \right\} =
\left\{ 
\begin{array}{ll}
1, & x \geq y, \\
c_1 c_2 x^2, & x < y,
\end{array}
\right.
$$
which is decreasing in $y$ for any fix $x$.

Let us assume that $\ell_0 = \underline V_0$, and it is sufficiently small.
The law of total probability and (\ref{1-step-dom-V}) imply
$\p \{ \l_1 \leq x \} \leq \p \{ \underline V_1 \leq x \}$. Assume that
for any $x > 0$,
$\p \{ \l_n \leq x \} \leq \p \{ \underline V_n \leq x \}$ for some $n \geq 1$.
Then
\begin{align*}
\p \{ \l_{n+1} \leq x \} &
= \int \p \{ \l_{n+1} \leq x | \l_n = y \} \d \p\{ \l_n \leq y\} \\
& \leq  \int \p \{ \underline  V_{n+1} \leq x | \underline V_n = y \}
\d \p\{ \l_n \leq y\}  \\
& \leq \int \p \{ \underline  V_{n+1} \leq x | \underline V_n = y \}
\d \p\{ \underline V_n \leq y\} \\
& =  \p \{ \underline  V_{n+1} \leq x  \},
\end{align*}
where we used the law of total probability, (\ref{1-step-dom-V}), the
induction hypothesis, the monotonicity of the conditional probabilities,
and that for two distribution functions $F,G$, such that
$F(x) \leq G(x)$, and for a monotone decreasing function $f$ we have
$\int f \d F \leq \int f \d G$ (\cite{SS} chapter 1.A). So we proved that
$\underline V_n  \stleq \l_n$ for every $n$.

For the asymptotic behaviour of $\underline V_n$ we have
$$
\p \left\{ \sqrt{c_1 c_2} \sqrt{n} \underline V_n > x \right\}
\to \e^{- x^2},
$$
and since $\underline V_n  \stleq  \l_n$, we obtain
$$
\liminf_{n \to \infty} \p \left\{  \sqrt{c_1 c_2} \sqrt{n} (m_n - \rho_k) > x
\right\} \geq \e^{-x^2}.
$$
In particular we have
\begin{align*}
\E \left[  \sqrt{c_1 c_2} \sqrt{n} (m_n - \rho_k) \right]
& = \int_0^\infty \p \left\{  \sqrt{c_1 c_2} \sqrt{n} (m_n - \rho_k) > x
\right\}\d x \\
& \geq  \int_0^\infty \p \left\{ \sqrt{c_1 c_2} \sqrt{n} \underline V_n > x \right\} \d x  \\
& \to \int_0^\infty \e^{-x^2} \d x,
\end{align*}
where at the last convergence we used the uniform integrability of $\sqrt{n} \underline V_n$.

\medskip

\textbf{Upper bound.}
Now we turn to the construction of the upper bound process. If the random point
falls in the change region $\mathcal{R}_i(n)$ then we have
$m_i(n+1) = m_i(n) - h_{n+1}^i (m_i(n) - \rho_k)$, and the other heights may change
or may not. In any case $\m_{n+1} \leq \m_n - \ev_i h_{n+1}^i (m_i(n) - \rho_k)$, where $\ev_i$ is the $i$th standard, $k$ dimensional unitvector. The probability
of this event is $|\mathcal{R}_i(n)|/|K_n|$ for which by (\ref{eq:A_n-bound}) and (\ref{eq:ter-est})
\[
\frac{|\mathcal{R}_i(n)|}{|K_n|} \geq \frac{\delta_1}{\pi} (m_i(n) - \rho_k)_+^2 =: c_3 (m_i(n) - \rho_k)_+^2.
\]
Instead of $h^{i}$ we put the stochastically smaller $h$, for which $1-h$ has distribution
function $H_*$ defined in (\ref{h-minor}). Note that for this $h$ we have
$\p \{ h = 0 \}= 1 - \delta_1$. We define the $k$-dimensional
process $\widehat \m_n$ as follows. Let
$i \in \{1,2, \ldots, k  \}$ such that $n+1 \equiv i \; (\mathrm{mod} \  k)$. Then define
\begin{equation} \label{hat}
\widehat \m_{n+1} = \left\{
\begin{array}{ll}
\widehat \m_n - \ev_i h_{n+1} \left( \widehat m_i(n) - \rho_k \right) &
\textrm{w.p. } c_3 (m_i(n) - \rho_k)_+^2, \\
\widehat \m_n, & \textrm{w.p. } 1 - c_3 (m_i(n) - \rho_k)_+^2,
\end{array} \right.
\end{equation}
where $h_1, h_2, \ldots$ are iid and $1 - h$ has distribution function  $H_*$ in  (\ref{h-minor}),
that is in each step at most one component decreases, and component $i$
can decrease only in steps $\ell k +i$, $\ell \in \N$.
From the construction it is clear that for each $\y \in \R^{k}$,
and for each upper set $U$
\begin{equation} \label{1-step-maj}
\p \{ \m_{n+ 1} \in U | \m_n = \y \} \leq 
\p \{ \widehat \m_{n+ 1} \in U | \widehat  \m_n = \y \}.
\end{equation}
Now we show that $\p \{ \widehat \m_{n+ 1} \in U | \widehat  \m_n = \y \}$ is a monotone
increasing function of $\y$ for any fixed upper set $U$. To do so, let
$n+1 \equiv i \; (\mathrm{mod} \  k)$, and define
$u_i(\y) = \inf \{u : (y_1, \ldots, y_{i-1}, u, y_{i+1}, \ldots, y_{k}) \in U \}$.
We may assume that $\y \in U$, $m_i(n) > \rho_k$ and $u_i(\y) > \rho_k$, otherwise
the statement is obvious.
Recall that in one step only coordinate $i$ can change, and so by (\ref{hat}) we have
\begin{align*}
\p \{ \widehat \m_{n+ 1} \in U | \widehat  \m_n = \y \}
& = 1 - \p \{ \widehat \m_{n+ 1} \not\in U | \widehat  \m_n = \y \} \\
& = 1 - \p \{ \widehat m_i(n+ 1) < u_i(\y) | \widehat  \m_n = \y \} \\
& = 1 - c_3 (y_i - \rho)^2 
\p \left\{ 1 - h < \frac{u_i(\y) - \rho}{y_i - \rho} \right\}\\
& = 1 - c_3 (u_i(\y) - \rho)_+^2.
\end{align*}
By the properties of the upper set we have that
$\y \leq \y^\prime \Rightarrow u_i(\y) \geq u_i(\y^\prime)$ and so the conditional
probability is monotone increasing.
As in the case of the lower estimation this allows us to prove the majorization
$\m_n \stleq \widehat \m_n$ as follows: If $\m_0 = \widehat \m_0$ in distribution,
then we have the majorization for $n=1$, and if it is true for some $n \geq 1$,
then for any upper set $U$
\begin{align*}
\p \{ \m_{n+1} \in U \} &
= \int \p \{ \m_{n+1} \in U | \m_n = \y \} \d \p\{ \m_n \leq \y \} \\
& \leq  \int \p \{ \widehat \m_{n+1} \in U | \widehat \m_n = \y \}
\d \p\{ \m_n \leq \y \}  \\
& \leq \int \p \{ \widehat \m_{n+1} \in U | \widehat \m_n = \y \}
\d \p\{ \widehat \m_n \leq \y \} \\
& =  \p \{ \widehat \m_{n+1} \in U  \},
\end{align*}
where we used the law of total probability, (\ref{1-step-maj}), the
induction hypothesis, the monotonicity of the conditional probabilities,
and that for two distribution functions $F,G$, such that
$F(\x) \leq G(\x)$ (understood componentwise),
and for a monotone increasing function $f$ we have
$\int f \d F \leq \int f \d G$ (\cite{SS}, chapter 6.B).

Putting
\begin{equation} \label{eq:overline-h}
\overline H(x) = \min \{ c_3 x^2, 1 \},
\end{equation}
as before we see that
\[
\widehat m_i(n) - \rho = \min \{ W_{i,j} : j \leq \lfloor n/k \rfloor \}, 
\]
where $\{ W_{i,j} : i=1,2, \ldots, k; j \in \N \}$ are iid random
variables with distribution function $\overline H$. We have
\[
\begin{split}
& \p \left\{ \sqrt{c_3 n} \max_{1 \leq i \leq k }
\min  \{ W_{i,j} : j \leq \lfloor n/k \rfloor \} \leq x \right\} \\
& = \left[ 1 - \p \left\{ \sqrt{c_3 n}
\min  \{ W_{i,j} : j \leq \lfloor n/k \rfloor \} > x \right\}
\right]^{k} \\
& = \left[ 1 -  \left( 1- \frac{x^2}{n} \right)^{\lfloor n/k \rfloor} \right]^{k} \\
& \to \left( 1 - \e^{- \frac{x^2}{k} }\right)^{k}.
\end{split}
\]
This, together with the stochastic majorization
$\m_n \stleq \widehat \m_n$ implies that
\[
\limsup_{n \to \infty}
\p \left\{ \sqrt{c_3 n } (m_n - \rho_{k} ) > x \right\} \leq
1 -  \left( 1 - \e^{- \frac{x^2}{k} }\right)^{k}.
\]

In particular we have
\begin{align*}
\E \left[  \sqrt{c_3 n} (m_n - \rho_{k}) \right]
& = \int_0^\infty \p \left\{  \sqrt{c_3 n} (m_n - \rho_{k}) > y
\right\}\d y \\
& \leq  \int_0^\infty \p \left\{ \sqrt{c_3 n} (\widehat m_n - \rho_{k} > y \right\} \d y  \\
& \to \int_0^\infty \left[ 1 -  \left( 1 - \e^{- \frac{y^2}{k} }\right)^{k} \right] \d y.
\end{align*}
\end{proof}

\section{Concluding remarks}

The major difference between regular polygons with odd and even number of vertices hides in the fact that while in the odd
case the change regions are always triangles, in the even case change
regions might be trapezoids or (in the degenerated case) triangles, hence their area might be of different order (see Figure~\ref{figure:caps}).
 We conjecture that in the latter case the `typical' change regions
are trapezoids, which would imply that the speed of the process is $1/n$. (Compare with Theorem~\ref{mainodd}, where we 
obtained $1/\sqrt{n}$ for the speed in the odd case.)
This conjecture is well supported by numerical experiments. We conclude the paper with the results of some computer simulations, see Figure~\ref{fig:7} and Figure~\ref{fig:8}. It is transparent that $\sqrt{n}$ and $n$ are the right normalizations, respectively.

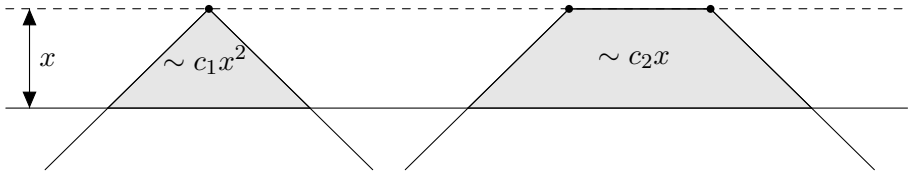
\begin{figure}
\begin{tikzpicture}[line cap=round,line join=round,>=triangle 45, x=0.45cm,y=1.0cm]
\clip(-2.92,0.44) rectangle (23.72,4.52);
\fill[fill=black,fill opacity=0.1] (3.08,3) -- (0.09,1.68) -- (6.07,1.68) -- cycle;
\fill[fill=black,fill opacity=0.1] (13.72,3) -- (10.73,1.68) -- (20.89,1.68) -- (17.9,3) -- cycle;
\draw [dash pattern=on 3pt off 3pt,domain=-2.92:23.72] plot(\x,{(--44.46-0*\x)/14.81});
\draw (3.08,3)-- (-1.76,0.86);
\draw (3.08,3)-- (7.92,0.86);
\draw (13.72,3)-- (8.88,0.86);
\draw (17.9,3)-- (22.74,0.86);
\draw [domain=-2.92:23.72] plot(\x,{(--24.9-0*\x)/14.82});
\draw (13.72,3)-- (17.9,3);
\draw [->] (-2.22,1.68) -- (-2.22,3);
\draw [->] (-2.22,3) -- (-2.22,1.68);
\draw (3.08,3)-- (0.09,1.68);
\draw (0.09,1.68)-- (6.07,1.68);
\draw (6.07,1.68)-- (3.08,3);
\draw (13.72,3)-- (10.73,1.68);
\draw (10.73,1.68)-- (20.89,1.68);
\draw (20.89,1.68)-- (17.9,3);
\draw (17.9,3)-- (13.72,3);
\fill [color=black] (3.08,3) circle (1.5pt);
\fill [color=black] (17.9,3) circle (1.5pt);
\draw[color=black] (2.97,2.32) node {$\sim c_1x^2$};
\fill [color=black] (13.72,3) circle (1.5pt);
\draw[color=black] (15.64,2.32) node {$\sim c_2x$};
\draw[color=black] (-1.68,2.32) node {$x$};
\end{tikzpicture}
\caption{ \label{figure:caps} Change regions of different shapes}

\end{figure}

%In Figure \ref{fig:7} one can find the different values of $\sqrt{n} (m_{n} - \rho_7)$ of 200 outcomes of the heptagon process after 100 
%steps ($n=100$), while in Figure \ref{fig:8} there are the values of $n (m_n - \rho_8)$ of 200 outcomes of the octagon process for $n=100$.

\begin{figure} 
\begin{center}
\includegraphics[scale=0.25]{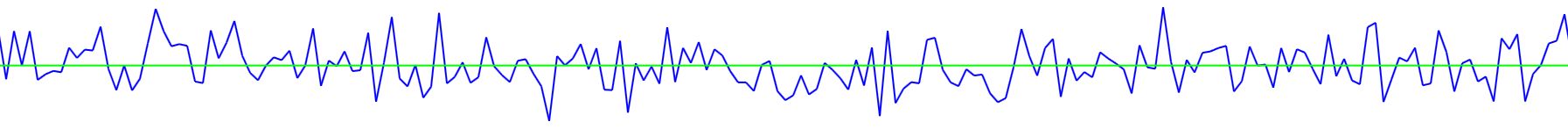}
\caption{ \label{fig:7} $\sqrt{100} (m_{100} - \rho_7)$ for 200 outcomes of the heptagon process (the minimal value is $0.215$, the maximal value is $1.078$)}
\end{center}

\end{figure}

\begin{figure}
\begin{center}
\includegraphics[scale=0.25]{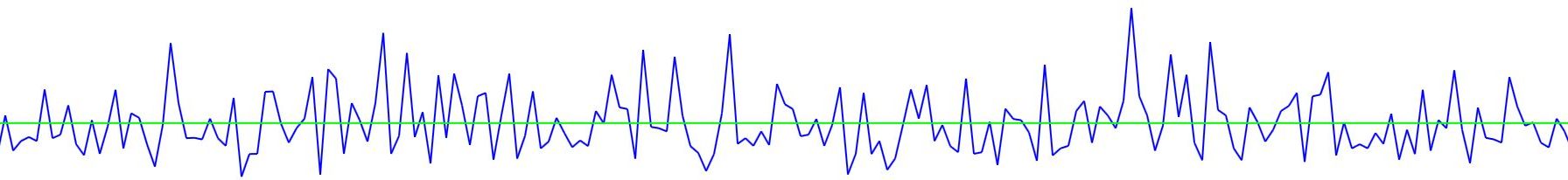} 
\caption{\label{fig:8} $100 (m_{100} - \rho_8)$ for 200 outcomes of the octagon process (the minimal value is $0.236$, the maximal value is $5.381$)}
\end{center}
 
\end{figure}

\section*{Appendix}

\begin{proof*}{Lemma \ref{lemma:konnyu}}
To prove part (i) note that the distributional equality means
\[
F(x) = F(a) F(x/a),
\]
for all $0 \leq x \leq a$. The monotonicity of $F$ easily implies that
the solution has the stated form for some $\delta > 0$.
% Here one ends up with the Cauchy-functional equation $g(x) + g(y) = g(x+y)$, for
% all $x \geq \eta > 0$, and $y \geq 0$. It is easy to show that any nice solution is 
% linear.
The `if' part follows by simple calculation.

We turn to part (ii). For any $x \in [1/2, 1]$
\[
\p \{ I(X \leq 1/2) = 0, \ \max \{ X, 1- X \}  > x \} = \p \{ X > x \} = 1 - F(x),
\]
and
\[
\begin{split}
& \p \{ I(X \leq 1/2) = 0 \} \, \p \{ \max \{ X, 1- X \}  > x \} \\
& = 
\left( 1 - F(1/2) \right) \left( 1 - F(x) + F(1-x-) \right).
\end{split}
\]
Solving the equation for $F$ we obtain the statement.
\end{proof*}

\begin{proof*}{Lemma \ref{lemma:recursion}}
After some calculation one obtains that given $\ell_n$ and $a_n$,
$\ell_{n+1} \stackrel{\mathcal{D}}{=} \min \{ \ell_n, Y \}$, where
$Y$ is a nonnegative random variable, such that $Y^\delta$ is uniformly
distributed on $[0,a_n/c]$.

% We say that $X$ \emph{stochastically dominates} (or just dominates) $Y$,
% in notation $X \geq_{\mathcal{D}} Y$, if
% $\p \{ X > x \} \geq \p \{ Y > x \}$ for all $x \in \R$.

For any $\varepsilon \geq 0$ let $U^{(\varepsilon)}, U^{(\varepsilon)}_1,
U^{(\varepsilon)}_2, \ldots$ be iid nonnegative random variables, such that
\[
\p \{ U^{(\varepsilon)} \leq x \} = x^\delta \frac{c}{a+\varepsilon},
\quad x \in \left[ 0, [(a+\varepsilon)/c]^{1/\delta} \right],
\]
that is $(U^{(\varepsilon)})^\delta \sim$ Uniform$[0, (a+\varepsilon)/c]$.
Put 
\[
M_n^{(\varepsilon)} = \min \{ U^{(\varepsilon)}_1, U^{(\varepsilon)}_2,
\ldots, U^{(\varepsilon)}_n \}.
\]

\noindent Since $a_n$ is decreasing, $Y \stgeq U^{(0)}$, therefore
\[
\ell_n \stgeq  \min \{ U^{(0)}_1, U^{(0)}_2, \ldots, U^{(0)}_n \}= M^{(0)}_n.
\]
As
\[
\p \left\{ \left( \frac{c n}{a} \right)^{1/\delta} M_n^{(0)} > x \right\}
\to \e^{-x^{\delta}},
\]
we have
\[
\liminf_{n \to \infty} \p \left\{ 
 \left( \frac{c n}{a} \right)^{1/\delta} \ell_n > x \right\} \geq
\lim_{n \to \infty} \p \left\{ \left( \frac{c n}{a} \right)^{1/\delta} M_n^{(0)} > x \right\}
= \e^{-x^{\delta}}.
\]
 
To prove the reverse inequality, let us fix $\varepsilon > 0$, $\beta > 0$.
Given that $a_n \leq a + \varepsilon$ we have
$Y \stleq  U^{(\varepsilon)}$, and thus given that
$a_{\lfloor \beta n \rfloor} < a + \varepsilon$ we have
\[
\ell_n \stleq  
\min \left\{ U^{(\varepsilon)}_1, U^{(\varepsilon)}_2, \ldots, U^{(\varepsilon)}_{\lfloor  (1-\beta)n \rfloor } \right\}
=: M^{(\varepsilon)}_{\lfloor (1-\beta)n \rfloor }.
\] 
By the law of total probability
\[
\begin{split}
& \p \left\{  \left( \frac{c n}{a} \right)^{1/\delta} \ell_n > x \right\} \\
& = \p \left\{ \left( \frac{c n}{a} \right)^{1/\delta} \ell_n > x
\Big| a_{\lfloor  \beta n \rfloor } < a+ \varepsilon \right\}
\cdot \p \left\{  a_{\lfloor  \beta n \rfloor } < a+ \varepsilon \right\} \\
& \phantom{=} + \p \left\{ \left( \frac{c n}{a} \right)^{1/\delta} \ell_n > x
\Big| a_{\lfloor  \beta n \rfloor } \geq a+ \varepsilon \right\}
\cdot \p \left\{  a_{\lfloor  \beta n \rfloor } \geq a+ \varepsilon \right\} \\
& \leq 
\p \left\{  \left( \frac{c n}{a} \right)^{1/\delta} M^{(\varepsilon)}_{\lfloor  (1-\beta)n \rfloor } > x \right\}
+  \p \left\{  a_{\lfloor  \beta n \rfloor } \geq a+ \varepsilon \right\}.
\end{split}
\]
By the assumption $a_n \downarrow a > 0$ a.s., so the second term goes to 0, while
from extreme value theory (see e.g.~\cite{Billingsley}, p.192) we have
\[
\lim_{n \to \infty}
\p \left\{ \left( \frac{c}{a} n \right)^{1/\delta} M^{(\varepsilon)}_{\lfloor  (1-\beta)n \rfloor } > x \right\}
= \e^{- x^{\delta} \frac{a}{a+\varepsilon} ( 1 - \beta)},
\]
that is, by the stochastic dominance
\[
\limsup_{n \to \infty} \p \left\{ 
\left( \frac{c}{a} n \right)^{1/\delta} \ell_n > x \right\} \leq 
\e^{- x^{\delta} \frac{a}{a+\varepsilon} ( 1 - \beta)}.
\]
Since $\varepsilon > 0$ and $\beta > 0$ are as small  as we want, we obtain
\[
\limsup_{n \to \infty} \p \left\{ 
\left( \frac{c}{a} n \right)^{1/\delta} \ell_n > x \right\} \leq 
\e^{- x^{\delta}},
\]
and the convergence in distribution is proved.

Once we have the distributional convergence, to prove the moment convergence
it is enough to show that $\{ n^{\alpha / \delta} \ell_n^\alpha \}$ is
uniformly integrable (see e.g.~\cite{Billingsley} Theorem 25.12).
Since $a_n$ is bounded, for some $\eta > 0$ we have
$a_n \leq a + \eta$ a.s.~for all $n \geq 1$, and thus
$\ell_n \stleq  M_n^{(\eta)}$. Therefore
\[
\p \{ n^{1 / \delta} \ell_n > x \} \leq
\p \{ n^{1 / \delta} M^{(\eta)}_n > x \} =
\left( 1 - \frac{x^\delta}{n} \frac{c}{a+\eta} \right)^n \leq
\e^{- x^\delta \frac{c}{a+\eta}}, 
\]
and the uniform integrability follows.
\end{proof*}

\bigskip

\textbf{Acknowledgement.} We are thankful to \'Arp\'ad Kurusa for providing exceeding computer simulations of the process and for the insightful figures.

\end{document}